\documentclass[12pt]{article}

\newcommand{\x}{x_\infty}
\newcommand{\y}{y_\infty}

\usepackage{amsmath, amsthm, amssymb}
\usepackage{enumerate}
\usepackage[top=30truemm, bottom=30truemm, left=25truemm, right=25truemm]{geometry}
\usepackage[dvipdfm]{graphicx}
\usepackage{accents}
\usepackage{bm}
\usepackage{caption}
\usepackage{cases}
\usepackage{tikz-cd}
\usepackage{titlesec}

\makeatletter
\def\sbar{\accentset{{\cc@style\underline{\mskip9mu}}}}
\def\mbar{\accentset{{\cc@style\underline{\mskip12mu}}}}
\def\lbar{\accentset{{\cc@style\underline{\mskip18mu}}}}

\makeatother
\titleformat*{\section}{\normalfont\large\bfseries}

\makeatletter
\def\@seccntformat#1{\csname the#1\endcsname. }
\makeatother

\renewcommand{\labelenumi}{(\arabic{enumi})}

\newcommand{\Spec}{\operatorname{Spec}}
\newcommand{\Proj}{\operatorname{Proj}}
\newcommand{\sinfty}{\scalebox{0.6}[0.65]{$\infty$}}
\newcommand{\sminus}{\!\smallsetminus\!}
\newcommand{\Pinf}{\scalebox{0.7}[0.7]{$P_{\infty}\!$}}
\newcommand{\Qinf}{\scalebox{0.7}[0.7]{$Q_{\infty}\!$}}
\theoremstyle{definition}
\newtheorem{Def}{Definition}[section]
\newtheorem{Thm}[Def]{Theorem}
\newtheorem{Prop}[Def]{Proposition}
\newtheorem{Lem}[Def]{Lemma}
\newtheorem{Cor}[Def]{Corollary}
\newtheorem{Rmk}[Def]{Remark}
\newtheorem{Exp}[Def]{Example}

\title{{\bf Differential forms on the curves associated to Appell-Lauricella hypergeometric series and the Cartier operator on them}}
\author{Ryo Ohashi\thanks{Graduate School of Environment and Information Sciences, Yokohama National University.
79-7 Tokiwadai, Hodogaya-ku, Yokohama 240-8501 Japan.
E-mail: \texttt{ohashi-ryo-hg@ynu.jp}}\; and Shushi Harashita\thanks{Graduate School of Environment and Information Sciences, Yokohama National University. 79-7 Tokiwadai, Hodogaya-ku, Yokohama 240-8501 Japan.
E-mail: \texttt{harasita@ynu.ac.jp}}}

\begin{document}
\maketitle
\begin{abstract}
Archinard studied the curve $C$ over $\mathbb{C}$ associated to an Appell-Lauricella hypergeometric series and differential forms on its desingularization. In this paper, firstly as a generalization of Archinard's results, we describe a partial desingularization of $C$ over a field $K$ under a mild condition on its characteristic and the space of global sections of its dualizing sheaf, especially we give an explicit basis of it. Secondly, when the characteristic is positive, we show that the Cartier operator on the space can be defined and describe it in terms of Appell-Lauricella hypergeometric series.
\medskip

\medskip

\end{abstract}

\section{Introduction}
An Appell-Lauricella hypergeometric series is defined to be a period of a family of degenerations of superelliptic curves,
see Section 2 for the details. We are interested in relations between the geometry of the associated family of the (possibly singular) curves and the analysis of Appell-Lauricella hypergeometric series. Among them, we shall describe Cartier-Manin matrices of the curves in terms of Appell-Lauricella hypergeometric series (Theorem \ref{MainTheoremGeneralA_k}). For this, we need to find nice explicit basis of the space of regular differential forms on the curves.

Let us start with recalling the most classical case:  a relation between elliptic curves and Gauss' hypergeometric series.
Gauss' hypergeometric series is defined to be
\begin{equation}
F(a,b,c\,;z) := \sum_{n=0}^{\infty} \frac{(a\,;n)(b\,;n)}{(c\,;n)(1\,;n)}z^n,
\end{equation}
where $(x\,;n) := \prod_{j=1}^n (x+n-j)$ and $a,b,c \in \mathbb{C},-c \notin \mathbb{N}$. It is well-known that $F(a,b,c\,;z)$ satisfies the differential equation $\mathcal{D}F(a,b,c\,;z) = 0$ with
\begin{equation}
	\mathcal{D} = z(1-z)\frac{d^2}{dz^2} + \bigl(c-(a+b+1)z\bigr)\frac{d}{dz} - ab.
\end{equation}
According to Deuring \cite[\S 7]{Deuring}, the elliptic curve $E: y^2=x(x-1)(x-z)$ in characteristic $p>0$ are supersingular if and only if $H_p(z) = 0$, where
\[
	H_p(z) := \sum_{i=0}^{m} \binom{m}{i}^{\!\!2} z^i
\]
with $m = (p-1)/2$, also see \cite[Theorem\,V.4.1]{Silverman}.
\cite{Igusa}, Igusa proved that $H_p(z)$ is a separable polynomial, by using the fact
that $H_p(z)$ satisfies the differential equation ${\mathcal{D}}'H_p(z) = 0$ with
\begin{equation}
	{\mathcal{D}}' = z(1-z)\frac{d^2}{dz^2} + (1-2z)\frac{d}{dz} - \frac{1}{4}.
\end{equation}
Note that (1.3) is (1.2) for $(a,b,c)=(1/2,1/2,1)$, and $H_p(z)$ is obtained by truncating the series $F(1/2,1/2,1\,;z)$ by degree $(p-1)/2$.
Also over $\mathbb{C}$, 
periods of the elliptic curve $y^2=x(x-1)(x-\lambda)$
are known to be described in terms of the hypergeometric series $F(1/2,1/2,1\,;\lambda)$
(cf. \cite[Chap.\;9, (6.1)]{Husemoeller}).
Some variants of hypergeometric series
have been used sporadically in some cases of higher genera,
for example see \cite[1.4]{IKO}, \cite[Section 2]{Brock},
\cite{Varchenko} and \cite[Section 3]{OKH}.

In this paper, we study the curve $C$ associated to an Appell-Lauricella hypergeometric series (see Definition \ref{Def:AL-hypergeometric series}), which is a certain generalization of Gauss' hypergeometric series. 
The curve $C$ is defined by the affine equation
\begin{equation}\label{eq:AL-curves}
	C: y^N = f(x):=\prod_{i=0}^r (x-\lambda_i)^{A_i}, \quad i \neq j \,\Rightarrow \lambda_i \neq \lambda_j, \quad \lambda_0, \ldots, \lambda_r \in \sbar{K}.
\end{equation}
The central purpose of this paper is to generalize the above result for elliptic curves to that for $C$ and partial desingularizations of $C$ under the assumption that $C$ is irreducible (Lemma \ref{lem:irreducibility}).

Archinard \cite[Section 2]{Archinard} gave the desingularization $X$ of $C$, especially for $K=\mathbb{C}$ and studied the space of regular differential forms on $X$. 
For our purpose, we first generalize the result to the case of a field $K$ whose characteristic is not a divisor of $N$.
In Section 2, we review basic properties of the curve $C$ associated to an Appell-Lauricella hypergeometric series, and construct the explicit desingularization of $C$. In Section 3, we describe the space of regular differential forms on it (Theorem \ref{TheoremSection3}).

As well as nonsingular curves, we often need to study singular curves. We shall deal with $C$ itself in Section 4 and the partial desingularization only at $\infty$ in Section 5. In particular, the latter objects (written as $\widetilde{C}$) appear when we consider degenerations of hyperelliptic curves or superelliptic curves. We shall give an explicit basis of the regular differential module on $C$ and $\widetilde{C}$ (Corollaries \ref{CorollarySection4} and \ref{cor:basis_for_tilde_C}), where the notion of the regular differential forms of singular curves is defined by Serre \cite[Section\,IV.3]{Serre}, see Section 4 for the details.
For separable $f(x)$, i.e., the case of $A_0=\cdots=A_r=1$ in \eqref{eq:AL-curves}, see Gonz\'{a}lez \cite[Proposition 3.1]{Gonz} for prime $N$ and Sutherland \cite[Lemma 6]{Sutherland} for general $N$.

As an application, in Section 6 we introduce the modified Cartier operator on the regular differential modules on $X$, $C$ and $\widetilde{C}$ (Theorem \ref{CartierOperatorOnC}). The last aim of this paper is to describe an relation between Appell-Lauricella hypergeometric series and the modified Cartier operator (Theorem \ref{MainTheoremGeneralA_k}). This result is a generalization of the fact that the coefficient of $x^{p-1}$ in the polynomial $\{x(x-1)(x-z)\}^{(p-1)/2}$ is equal to a truncation of $(-1)^{(p-1)/2}F(1/2,1/2,1\,;z)$.
The research of Cartier-Manin matrices (or their dual notion: Hasse-Witt matices)
have long history. Among them,
Sutherland gave a fast algorithm computing
Cartier-Manin matices of superelliptic curves \cite{Sutherland},
also see Harvey-Sutherland \cite{HS} for hyperelliptic curves.
This paper gives a formula of Cartier-Manin matrices
whose entries are considered as polynomials in $\lambda_i$ of \eqref{eq:AL-curves}.
This result may not contribute to speeding up the computation
if $\lambda_i$ are constants, but would have many applications
if $\lambda_i$ are indeterminates.
Indeed, Cartier-Manin matrices with polynomial entries
are used in the papers \cite{KH} and \cite{KHS} and so on,
for enumeratations of superspecial curves and for a proof of the existence of
supersingular curves.

This paper is organized as follows. In Section 2, we study the fundamentals of curves associated to Appell-Lauricella hypergeometric series and give their explicit desingularizations. In Section 3, we describe the space of regular differential forms on them. In Section 4, we describe the space of regular differential forms on $C$. In Section 5, we describe the space of regular differential forms on $\widetilde{C}$. Finally in Section 6, we show that the modified Cartier-Manin operator stabilizes the space of regular differential forms on $\widetilde{C}$  and so on, and describe the relation between the Cartier operator and Appell-Lauricella hypergeometric series.

\subsection*{Acknowledgements}
The authors thank Yuya Yamamoto who pointed out some
mistakes in an earlier version.
This work was supported by JSPS Grant-in-Aid for Scientific Research (C) 17K05196 and 21K03159.
\newcommand{\cP}{{\mathcal P}}
\newcommand{\cPfin}{{\cP}_{\rm fin}}
\newcommand{\cPinf}{{\cP}_{\infty}}
\section{The curves associated to Appell-Lauricella hypergeometric series}
\setcounter{equation}{0}
In this section, we recall the definition of the curves associated to Appell-Lauricella hypergeometric series and properties of them. Let $K$ be a field.
\begin{Def}\label{Def:AL-curve}
Let $N$ be a positive integer which is not a multiple of the characteristic of $K$. A {\it curve associated to Appell-Lauricella hypergeometric series} is the 1-dimensional algebraic set defined by
\begin{equation*}
	C: y^N = f(x)
\end{equation*}
for an $f(x) \in K[x]$, which is possibly inseparable: $f(x)$ is factorized as
\[
	 f(x) = \prod_{i=0}^r (x-\lambda_i)^{A_i}, \quad \lambda_0,\ldots, \lambda_r \in \mbar{K},
\]
where $A_i \geq 1$ and $\lambda_i \neq \lambda_j$ for $i \neq j$.
\end{Def}
\begin{Rmk}
The curve $C$ above (or more precisely its desingularization) is called {\it superelliptic} if $A_i = 1$ for all $i \in \{0,\ldots,r\}$.
If $N = 2$ and $r>3$ in addition, the curve $C$ is called {\it hyperelliptic}. A curve associated to Appell-Lauricella hypergeometric series is a certain generalization of these curves.
\end{Rmk}
\begin{Def}\label{Def:AL-hypergeometric series}
Appell-Lauricella hypergeometric series is defined to be
\begin{equation}
	{\mathcal F}(a,b_1,\ldots,b_d,c\,;z_1,\ldots,z_d) := \sum_{n_1=0}^{\infty} \cdots \sum_{n_d = 0}^{\infty} \frac{(a\,;\sum{n_j})\prod(b_j\,;n_j)}{(c\,;\sum{n_j})\prod(1\,;n_j)}\prod_{j=1}^d z_j^{n_j},
\end{equation}
where $a,b_1,\ldots,b_d,c\in\mathbb{C},-c \notin \mathbb{N}$.
\end{Def}

It is obvious that ${\mathcal F}(a,b,c\,;z) = F(a,b,c\,;z)$ when $d = 1$, hence Appell-Lauricella hypergeometric series ${\mathcal F}(a,b_1,\ldots,b_d,c\,;z_1,\ldots,z_d)$ can be regarded as a certain generalization of Gauss' hypergeometric series $F(a,b,c\,;z)$. If $0<\mathfrak{R}a<\mathfrak{R}c$, then it is known that ${\mathcal F}(a,b_2,\ldots,b_r,c\,;\lambda_2,\ldots,\lambda_r)$ has the integral representation as below:
\begin{equation}
	{\mathcal F}(a,b_2,\ldots,b_r,c\,;\lambda_2,\ldots,\lambda_r) = \frac{\Gamma(c)}{\Gamma(a)\Gamma(c-a)}\int_1^{\infty} \prod_{i=0}^r (x-\lambda_i)^{-\mu_i}dx
\end{equation}
with $\lambda_0=0$ and $\lambda_1=1$, where $\mu_0=c-\sum_{j=2}^r{b_j}$,
$\mu_1=1+a-c$ and $\mu_j=b_j$ for $j=2,\ldots,r$.

In the case that all $\mu_i$ are positive rational numbers, by setting the integrand of (2.2) as $1/y$, we see that
the hypergeometric function is associated to the curve
\[
y^N = \prod_{i=0}^r (x-\lambda_i)^{A_i},
\]
where we set $N$ to be the least common multiple of the denominators of $\mu_0,\ldots,\mu_r$
and set $A_i = N\mu_i$ so that $(N, A_0,\ldots,A_r) = 1$ holds.

By \cite[Chap.\;VI, Theorem 9.1]{Lang}, we have a condition for $C$ in Definition \ref{Def:AL-curve} to be irreducible:
\begin{Lem}\label{lem:irreducibility}
$C$ is irreducible over $K$ (and over $\overline K$) if and only if $(N, A_0,\ldots,A_r) = 1$.
\end{Lem}

From now on, we suppose that $(N, A_0,\ldots,A_r) = 1$, since we are interested in the case that $C$ is irreducible. Now we regard $C$ as a projective variety in ${\mathbb P}^2 = \Proj K[x_0,x_1,x_2]$. Set $A_{\infty} := \bigl|N - \sum_{k=0}^r{A_k}\bigr|$. The projective equation of $C$ reads
\begin{itemize}
\item {\sl Case 1:} $N - \sum_{k=0}^r{A_k} > 0$. \quad ${x_2}^N = {x_0}^{A_{\sinfty}}\prod(x_1-\lambda_ix_0)^{A_i}$;
\item {\sl Case 2:} $N - \sum_{k=0}^r{A_k} < 0$. \quad ${x_2}^N{x_0}^{A_{\sinfty}} = \prod(x_1-\lambda_ix_0)^{A_i}$;
\item {\sl Case 3:} $N - \sum_{k=0}^r{A_k} = 0$. \quad ${x_2}^N = \prod(x_1-\lambda_ix_0)^{A_i}$.
\end{itemize}
Put $P_j = (1:\lambda_j:0)$ for $j \in \{0,\ldots,r\}$ and set ${\mathcal P}_{\rm fin}:=\{P_0, \ldots, P_r\}$. 
Put $P_\infty=(0:1:0)$ in {\sl Case 1} and $P_\infty=(0:0:1)$ in {\sl Case 2}.
We set $\cPinf:=\{P_\infty\}$ for {\sl Cases 1} and {\sl 2}
and $\cPinf:=\emptyset$ for {\sl Case 3}.
Set 
\[
\mathcal P = {\mathcal P}_{\rm fin} \cup {\mathcal P}_\infty.
\]
By the Jacobian criterion and the assumption that the characteristic of $K$ is not a divisor of $N$, it is straightforward to see where $C$ has singularities.
\begin{Lem}
Any singular point of $C$ belongs to ${\mathcal P}$. Moreover $C$ is singular at $P_i$ for each $i \in \{0,\ldots,r,\infty\}$ if and only if $A_i > 1$.
\end{Lem}

Next we review the explicit description of the desingularization $X$ of $C$
obtained by Archinard \cite[3.1]{Archinard}, which works also in positive characteristic. Let $g_i := (N,A_i)$ for each $i \in \{0,\ldots,r, \infty\}$, hence ${N}_{i} = N/{g_i}$ and ${A'}_{\!i} = {A_i}/{g_i}$ are coprime non-negative integers. Thus, there exist $m_i,n_i \in \mathbb{Z}$ such that ${m_i}{A'}_{\!i} + {n_i}{N}_{i} = 1$.
\begin{Prop}[Archinard {\cite[3.1.1]{Archinard}}]\label{DesingFinitePlace}
Let $j \in \{0,\ldots,r\}$ and put $f_j(x) := \prod_{i \neq j} (x-\lambda_i)^{A_i}$. Let $D(f_j)$ be the open subscheme obtained by excluding the part of $f_j(x) = 0$ from $\mathbb{A}^3 = \Spec K[x,u,z]$ and 
set $X_j$ to be the closed subscheme of $D(f_j)$ defined by
\[
	X_j: z^{{N}_{j}} = (x-\lambda_j)u^{m_j},\ u^{g_j} = f_j(x).
\]
Then $X_j$ is nonsingular with (birational) morphism
\[
	\pi_j: X_j \rightarrow 
C \sminus ({\mathcal P} \sminus \{P_j\})
\,;\, (x,u,z) \mapsto (1:x:u^{n_j}z^{{A'}_{\!j}}).
\]
Moreover $\pi_j$ induces an isomorphism $X_j\sminus\pi^{-1}(\{P_j\}) \overset{\!\cong}{\longrightarrow} C\sminus {\mathcal P}$, whose inverse is
\[
\rho_j: C\sminus {\mathcal P} \to X_j\sminus\pi^{-1}(\{P_j\})
\,;\,(1:x:y) \mapsto (x,y^{{N}_{j}}(x-\lambda_j)^{-{A'}_{\!j}},y^{m_j}(x-\lambda_j)^{n_j}).
\]
\end{Prop}
\begin{Prop}[Archinard {\cite[3.1.2]{Archinard}}]\label{DesingInfinitePlace}
Put $f_{\infty}(\xi) := \prod_{i=0}^r (1-\lambda_i\xi)^{A_i}$. We define $X_{\infty}$ and $\pi_{\infty}$ in each case as follows. Then $X_{\infty}$ is nonsingular, and birationally equivalent to $C$ under a rational map $\pi_{\infty}$.
\begin{itemize}
\item {\sl Case 1:} $N - \sum{A_k} > 0$. \quad  In this case, $\x := {x_0}/{x_1}$ and $\y := {x_2}/{x_1}$ are regular on $\infty$. Then we write $D(f_\infty(\x))$ be the open subscheme obtained by excluding the part of $f_\infty(\x) = 0$ from $\mathbb{A}^3 = \Spec K[\x,u,z]$ and set $X_{\infty}$ to be the closed subscheme of $D(f_\infty(\x))$ defined by
\[
	X_\infty : z^{{N}_{\sinfty}} = \x u^{m_{\sinfty}},\ u^{g_{\sinfty}} = f_{\infty}(\x).
\]
Then $X_{\infty}$ is nonsingular with (birational) morphism
\[
	\pi_{\infty}: X_{\infty} \rightarrow C\sminus {\mathcal P}_{\rm fin}\,;\,(\x,u,z) \mapsto (\x : 1 : u^{n_{\sinfty}}z^{{A'}_{\!\sinfty}}).
\]
which induces an isomorphism
$X_{\infty}\sminus\pi^{-1}(\cPinf) \overset{\cong}{\longrightarrow} C\sminus {\mathcal P}$, whose inverse is
\[
	\rho_{\infty} : C\sminus {\mathcal P} \rightarrow X_{\infty}\sminus\pi^{-1}(\cPinf)\,;\,(\x : 1 : \y) \mapsto (\x, \x^{-{A'}_{\!\sinfty}}\y^{{N}_{\sinfty}}, \x^{n_{\sinfty}}\y^{m_{\sinfty}}).
\]

\item {\sl Case 2:} $N - \sum{A_k} < 0$. \quad In this case, $\x := {x_0}/{x_2}$ and $\y := {x_1}/{x_2}$ are regular on $\infty$. Then we write $D(f_\infty(u))$ be the open subscheme obtained by excluding the part of $f_\infty(u) = 0$ from $\mathbb{A}^4 = \Spec K[u,v,w,z]$ and set $X_{\infty}$ to be the closed subscheme of $D(f_\infty(u))$ defined by
\[
	X_{\infty} : 
u=w^{m_{\!\sinfty}}z^{N_{\sinfty}},
\ z^{{A'}_{\!\sinfty}} = vw^{n_{\sinfty}},
\ w^{g_{\sinfty}} = f_{\infty}(u).
\]
Then $X_{\infty}$ is nonsingular with birational morphism
\[
	\pi_{\infty}: X_{\infty} \rightarrow C\sminus\cPfin\,;\,(u,v,w,z) \mapsto (vw^{m_{\sinfty}}z^{{N}_{\sinfty}}:v:1),
\]
which induces an isomorphism $X_{\infty}\sminus\pi^{-1}(\cPinf) \overset{\cong}{\longrightarrow} 
C\sminus \cP$, whose inverse is
\[
	\rho_{\infty} : C\sminus\cP \rightarrow X_{\infty}\sminus\pi^{-1}(\cPinf)\,;\,(\x : \y : 1) \mapsto (\x \y^{-1}, \y, \x^{{A'}_{\!\sinfty}}\y^{-{N}_{\sinfty}-{A'}_{\!\sinfty}}, \x^{n_{\sinfty}}\y^{m_{\sinfty}-n_{\sinfty}}).
\]
\end{itemize}
Note that in {\sl Case 3:} $N-\sum\hspace{-0.25mm}A_i = 0$,
the points at infinity $(0:1:\zeta)$ with $\zeta^N=1$ are nonsingular by Lemma 2.5. Hence we do not need to consider $X_\infty$ as
Archinard excluded {\sl Case 3}.
\end{Prop}

\begin{Rmk}
Note that $X_i$ does not depend on the choice of $m_i,n_i$ up to isomorphism. More precisely, suppose that $({m'}_{\!i},{n'}_{\!i})$ is another pair of integers satisfying ${m'}_{\!i}{A'}_{\!i} + {n'}_{\!i}{N}_{i} = 1$ and let ${X'}_{\!i}$ be the set obtained from $({m'}_{\!i},{n'}_{\!i})$ in the same way, then ${X'}_{\!i}$ is isomorphic to $X_i$.
We give a proof of the case $i \in \{0,\ldots, r\}$. Now there is a relation $({n_i}-{n'}_{\!i}){N}_{i} = -({m_i}-{m'}_{\!i}){A'}_{\!i}$, then we have $e := (n_i - {n'}_{\!i})/{A'}_{\!i} \in \mathbb{Z}$ since ${N}_{i}$ and ${A'}_{\!i}$ are coprime non-negative integers. Hence consider the morphism $X_i \rightarrow {X'}_{\!i}\,;\,(x,u,z) \mapsto (x,u,u^ez)$ which has the obvious inverse, thus we conclude that $X_i \cong {X'}_{\!i}$. Also the uniqueness of $X_{\infty}$ can be proved similarly.
\end{Rmk}

\begin{Rmk}\label{RemarkOnDesing}
The action on $C$ of the group $\mu_N$ of $N$-th roots defined by $(x,y) \to (x,\zeta y)$ for $\zeta\in\mu_N$ is extended to that on $X_j$ by $(x,u,z) \to (x,\zeta^{{N}_{j}}u, \zeta^{m_j}z)$. Also a similar thing holds for $P_\infty$.
\end{Rmk}
We define the desingularization $X$ obtained by gluing $X_0,\ldots,X_r$ and $X_\infty$ along
$X_i\sminus \pi^{-1}(\{P_i\})$ and $X_j\sminus\pi^{-1}(\{P_j\})$ via the isomorphisms
\[
	X_i \sminus \pi^{-1}(\{P_i\}) \overset{\pi_i}{\longrightarrow} C\sminus\cP\overset{\pi_j}{\longleftarrow} X_j\sminus\pi^{-1}(\{P_j\}).
\]
Then by gluing $\pi_i: X_i \rightarrow C$, we also obtain a morphism $\pi: X \rightarrow C$ such that $\pi|_{X_i} = \pi_i$ for all $i \in \{0,\ldots,r,\infty\}$. As can be found in \cite[Section 3.2]{Archinard}, $X$ is the desingularization of $C$ under $\pi$. The following genus formula of $X$
is shown in the same way as in the case of $K=\mathbb{C}$, \cite[Theorem 4.1]{Archinard}.
\begin{Thm}
The genus of $X$ is given by
\[
	g(X) = 1 + \frac{1}{2}\biggl(rN - \sum_{j=0}^r (N,A_j) - \Bigl(N, N-\sum_{k=0}^r A_k\Bigr)\biggr).
\]
\begin{proof}
Let $C \rightarrow \mathbb{P}^1$ be the projection $(x_0:x_1:x_2) \mapsto (x_0:x_1)$ except for $(0:0:1) \mapsto (0:1)$ in {\sl Case 2}. Composing it and $\pi: X \rightarrow C$, we obtain a finite separable morphism $X \rightarrow \mathbb{P}^1$.
\begin{center}
\begin{tabular}{|c||c|c|} \hline
point $P$ of $C$ & \# of $\pi$-preimages $Q$ of $P$ & ramification index at $Q$\\\hline\hline
$(1:\lambda_j:0)$ & $(N,A_j)$ & $N/(N,A_j)$ \\\hline
$\infty$ & $(N,N-\sum{A_k})$ & $N/(N,N-\sum{A_k})$ \\\hline
other points & 1 & 1\\\hline
\end{tabular}
\end{center}
The genus of a projective line $\mathbb{P}^1$ is 0, so directly we see that
\[
	\begin{aligned}
	2g(X) -2 & = -2N + \sum_{j=0}^r (N,A_j)\!\biggl(\frac{N}{(N,A_j)}-1\biggr) + \bigl(N,N-\sum A_k\bigr)\!\biggl(\frac{N}{(N,N-\sum{A_k})}-1\biggr)\\
	\!\!& = rN - \bigl(N,N-\sum A_k\bigr) - \sum_{j=0}^r (N,A_j).
	\end{aligned}
\]
This is the desired conclusion.
\end{proof}
\end{Thm}
\section{The space of regular differential forms on $X$}
\setcounter{equation}{0}
In this section, for the desingularization map $\pi: X \to C$ constructed as in Section 2, we describe the regularity condition of rational differential forms on $X$. This enables us to give an explicit basis of the space $\varOmega[X]$ of regular differential forms on $X$, where ``regular'' is often called ``of first kind". Note that $\varOmega[X]$ is realized as a subspace of the space $\varOmega(C)$ of rational differential forms on $C$.

A general idea to describe 
the space of differential forms on plane curves and the Cartier operator on it
is found in St\"ohr-Voloch \cite{SV}. As explained there,
Gorenstein \cite[Theorem 12]{Gorenstein} gives a description of regular differential forms on the projective nonsingular model of a plane curve $\varGamma$, but our case does not satisfy his assumption: $y$ is an integral element over $K(x)$. This would mean that the Zariski closure of $\varGamma$ in ${\mathbb P}^2$ is regular at every infinite place. So let us formulate a lemma, which works in our case. The proof was done in the proof of \cite[Theorem 12]{Gorenstein}.

\begin{Thm}[Gorenstein]\label{Gorenstein}
Let $\varGamma$ be a plane curve ${\rm Spec}\,R$ with $R=K[x,y]/(F)$ for an irreducible element $F$ of $K[x,y]$ of degree ${\rm m}$. Let $L$ be the function field of $\varGamma$ and $X$ the nonsingular projective curve having the same function field $L$. Assume that $x$ considered as an element of $L$ is transendental over $K$ and $y$ considered as an element of $L$ is separable over $K(x)$. A rational differential form $\omega$ is regular on $X$ if and only if it can be written in the form
\[
	\frac{\phi(x,y)}{(\partial F/\partial y)(x,y)} dx
\]
such that $\phi(x,y)$ is an {\it adjoint element}, whose precise meaning is
\begin{enumerate}
\setlength{\itemsep}{0cm}
\renewcommand{\labelenumi}{(\roman{enumi})}
\item $\phi(x,y)\in {\frak C}$, where $\frak C$ is the conductor of $R=K[x,y]/(F)$ in the integral closure $\mbar{R}$ of $R$ in $L$, i.e. $\frak C := \{z\in R\,;\,z {\mbar{R}} \subset R\}$.
\item $\phi'(x',y'){x'}^{{\rm m}-3-{\rm h}} \in {\frak C}'$, where $\phi'$ be the polynomial in $x'$ and $y'$ defined by $\phi(x,y) = \phi'(x',y')/({x'})^{\rm h}$ with $(x',y')=(1/x,y/x)$ and ${\rm h}=\deg\phi$, and $\frak C'$ is the conductor of $R'=K[x',y']/(F')$ with $F(x,y)=F'(x',y')/({x'})^{\rm m}$ in the integral closure $\mbar{R'}$ of $R'$ in $L$.
\item $\phi''(x'',y''){y''}^{{\rm m}-3-{\rm h}}\in {\frak C}''$, where $\phi''$ be the polynomial in $x''$ and $y''$ defined by $\phi(x,y)=\phi''(x'',y'')/({y''})^{\rm h}$ with $(x'',y'')=(x/y,1/y)$  and ${\rm h}=\deg\phi$, and $\frak C''$ is the conductor of $R''=K[x'',y'']/(F'')$ with $F(x,y)=F''(x'',y'')/({y''})^{\rm m}$ in the integral closure $\lbar{R''}$ of $R''$ in $L$.
\end{enumerate}
\end{Thm}
\begin{Rmk}
Here, $\phi(x,y) \in {\frak C}$ is equivalent to $\phi(x,y) \in {\frak C}_P := \{z\in R_P\,;\,z {\lbar{R_P}} \subset R_P\}$ for maximal ideal $P$ of $R$, where ${\lbar{R_P}}$ is the integral closure of $R_P$ in $L$. Moreover ${\frak C}_P ={\frak C}_P^* \cap R_P$, where ${\frak C}_P^*$ is the conductor of $(R_P)^{*}$ in $(\lbar{R_P})^{*}$, where $*$ means taking the completion (cf.\ \cite[Theorem 2]{Gorenstein}). If the Zariski closure of $\varGamma$ in ${\mathbb P}^2$ is regular at every infinite place, (i) and (ii) in Theorem \ref{Gorenstein} can be replaced by ${\rm h}\le {\rm m}-3$.
\end{Rmk}

Let us return to our case $F=y^N-f(x)$. Let $\mu_N$ be the subgroup of $\mbar{K}^\times\!$ consisting of $N$-th roots of unity. We consider the case that $K$ contains $\mu_N$ and $\{\lambda_0,\ldots,\lambda_r\}$.
Let $\phi(x,y)$ be an adjoint element in the sense of the definition in Theorem \ref{Gorenstein}. Using $y^N=f(x)$, one can write
\[
	\phi(x,y) = \phi_0(x) + \phi_1(x) y + \cdots + \phi_{N-1}(x) y^{N-1}
\]
with $\phi_j(x) \in K[x]$. Consider the action on $R$ of the group $\mu_N$ consisting of $N$-th roots of unity by $(x,y) \mapsto (x,\zeta y)$ for $\zeta\in \mu_N$. As this action stabilizes $\frak C$, $\frak C'$ and $\frak C''$, we have that $\phi(x,\zeta y)$ is also an adjoint element for all $\zeta\in \mu_N$. This implies that each term $\phi_j(x)y^j$ is an adjoint element. Clearly $\phi_j(x)$ is uniquely written as $\varphi(x)\prod_{i=0}^r(x-\lambda_i)^{a_i}$, where $\varphi(x)$ is coprime to $x-\lambda_i$ for $i=0,1,\ldots,r$. As we can check $\phi_j(x)y^j \in {\frak C}$ by looking at whether $\phi_j(x)y^j$ in ${\frak C}_P$ for all $P \in \{P_0,\ldots,P_r\}$, we conclude that $\varphi(x)\prod_{i=0}^r(x-\lambda_i)^{a_i}y^j\in {\frak C}$ if and only if $\prod_{i=0}^r(x-\lambda_i)^{a_i}y^j\in {\frak C}$. Also at an infinite place, the condition (ii) and (iii) of Theorem \ref{Gorenstein} is described as the degree of $\phi_j$ is less than or equal to a certain constant depending only on $C$ (cf.\ the proof of Proposition \ref{Prop_X_infinite} below). Hence, if $\varphi(x)\prod_{i=0}^r(x-\lambda_i)^{a_i}y^j$ is an adjoint element, then $\prod_{i=0}^r(x-\lambda_i)^{a_i}y^j\in {\frak C}$ is an adjoint element. Thus we conclude
\begin{Lem}
The regular differential module $\varOmega[X]$ has a basis consisting of elements of the form
\[
	\omega_{(s,\bm{a})} := \frac{\prod_{i=0}^r(x-\lambda_i)^{a_i}}{y^s}dx, \quad 0 \leq s \leq N-1
\]
where $\bm{a} = (a_0, \ldots,a_r)$ with $a_i \geq 0$.
\end{Lem}
Note that $\pi^{*}\omega_{(s,\bm{a})}$ is regular at every finite place except $Q_i \in \pi^{-1}(\{P_i\})$ for $i = 0,\ldots,r$. Let us find the condition that $\pi^{*}\omega_{(s,\bm{a})}$ is regular at $Q_i$ and at infinite places.
\begin{Prop}\label{Prop_X_finite}
For $j \in \{0,\ldots,r\}$, the pull-back $\pi^{*}\omega_{(s,\bm{a})} \in \varOmega(X)$ is regular at $Q_j \in \pi^{-1}(\{P_j\})$ if and only if
\[
	a_j \geq \frac{sA_j + (N,A_j)}{N} - 1.
\]
\begin{proof}
The equations defining $X_j$ gives other equations\vspace{-1mm}
\begin{align*}
	{N}_{j} z^{{N}_{j}-1}dz &= u^{m_j}dx + m_j(x-\lambda_j)u^{m_j-1}du,\\
	g_ju^{g_j-1}du &= (df_j/dx)dx.\\[-7mm]
\end{align*}
By $\pi^{*}(x) = x$ and $\pi^{*}(y)=u^{n_j}z^{{A'}_{\!j}}$ with $\pi^*(x-\lambda_j) = u^{-m_j}z^{N_j}$
, a direct calculation shows
\begin{equation}\label{Eq:ProofOfProp3.4}
	\pi^{*}\omega_{(s,\bm{a})} = \frac{Nu^{g_j-sn_j-(1+a_j)m_j}z^{-s{A'}_{\!j}+(1+a_j)N_j-1}\prod_{i\ne j}(x-\lambda_i)^{a_i} }{g_jf_j(x) + m_j(x-\lambda_j)(df_j/dx)}dz.
\end{equation}
Hence $\pi^{*}\omega_{(s,\bm{a})}$ is regular at $Q_j$ if and only if $-s{A'}_{\!j}+(1+a_j){N}_{j}-1 \geq 0$.
\end{proof}
\end{Prop}
Similarly, the regularity at the fiber of the infinity is described as below:
\begin{Prop}\label{Prop_X_infinite}
The pull-back $\pi^{*}\omega_{(s,\bm{a})} \in \varOmega(X)$ is regular at every $Q_{\infty} \in \pi^{-1}(\{P_{\infty}\})$ in {\sl Cases 1} and {\sl 2} and at every $Q_\infty \in \pi^{-1}(\{(0:1:\zeta); \zeta^N=1\})$ in {\sl Case 3} if and only if
\[
	\sum_{k=0}^r a_k \leq \frac{s\sum{A_k} - (N, N-\sum{A_k})}{N} - 1.
\]
\begin{proof}
In each case, we can write $\omega_{(s,\bm{a})} = {x_0}^{s-\sum\hspace{-0.1mm}{a_k}-2}{x_2}^{-s}\prod(x_1-\lambda_ix_0)^{a_i}(x_0dx_1-x_1dx_0)$.
\begin{itemize}
\item {\sl Case 1:} $N - \sum{A_k} > 0$. \quad In this case, recall that $\x = {x_0}/{x_1}$ and $\y = {x_2}/{x_1}$;
\[
	\omega_{(s, \bm{a})} = - \x^{s-\sum\hspace{-0.1mm}{a_k}-2}\y^{-s}\prod(1-\lambda_i\x)^{a_i}d\x.
\]
The equations defining $X_{\infty}$ gives other equations\vspace{-1mm}
\begin{align*}
	N_{\infty} z^{{N}_{\sinfty}-1}dz &= u^{m_{\sinfty}}d\x + m_{\infty}\x u^{m_{\sinfty}-1}du,\\
	g_{\infty}u^{g_{\sinfty}-1}du &= (df_{\infty}/d\x)d\x.
\end{align*}
Now recall that $\pi^{*}(\x) = \x$ and $\pi^{*}(\y)=u^{n_{\sinfty}}z^{{A'}_{\!\sinfty}}$, then we obtain
\begin{equation}\label{Eq:3.2}
	\pi^{*}\omega_{(s,\bm{a})} = \frac{-Nu^{g_{\sinfty}-s(m_{\sinfty}+n_{\sinfty})+(1+\sum\hspace{-0.1mm}{a_k})m_{\sinfty}}z^{s({N}_{\sinfty}-{A'}_{\!\sinfty})-(1+\sum\hspace{-0.1mm}{a_k}){N}_{\sinfty}-1}\prod(1-\lambda_i\x)^{a_i}}{g_{\infty}f_{\infty}(\x) + m_{\infty}\x(df_{\infty}(\x)/d\x)}dz.
\end{equation}
Hence $\pi^{*}\omega_{(s,\bm{a})}$ is regular at $Q_{\infty}\hspace{-0.3mm}$ if and only if $s({N}_{\infty}-{A'}_{\!\infty})-(1+\sum\hspace{-0.1mm}{a_k}){N}_{\infty}-1 \geq 0$.

\item {\sl Case 2:} $N - \sum{A_k} < 0$. \quad In this case, recall that $\x = {x_0}/{x_2}$ and $\y = {x_1}/{x_2}$;
\[
	\omega_{(s,\bm{a})} = \x^{s-\sum\hspace{-0.1mm}{a_k}-2}\prod(\y - \lambda_i\x)^{a_i}(\x d\y - \y d\x).
\]
The equations defining $X_{\infty}$ give other equations
\begin{align*}
	{A'}_{\!\infty} u^{{A'}_{\!\sinfty}-1}du &= {N}_{\infty} v^{{N}_{\sinfty}-1}wdv + v^{{N}_{\sinfty}}dw,\\
	{A'}_{\!\infty} z^{{A'}_{\!\sinfty}-1}dz &= w^{n_{\sinfty}}dv + n_\infty vw^{n_{\sinfty}-1}dw,\\
	g_{\infty}w^{g_{\sinfty}-1}dw &= (df_{\infty}(u)/du) du.
\end{align*}
A tedious computation with these equations shows
\begin{equation}\label{Eq:3.3}
	\pi^{*}\omega_{(s,\bm{a})} = \frac{Nw^{g_{\sinfty}+s(m_{\sinfty}-n_{\sinfty})-(1+\sum\hspace{-0.1mm}{a_k})m_{\sinfty}}z^{s({N}_{\sinfty}+{A'}_{\!\sinfty})-(1+\sum\hspace{-0.1mm}{a_k}){N}_{\sinfty}-1}\prod (1-\lambda_iu)^{a_i}}{m_\infty u(df_{\infty}(u)/du)-g_{\infty}f_{\infty}(u)}dz.
\end{equation}
Hence $\pi^{*}\omega_{(s,\bm{a})}$ is regular at $Q_\infty\hspace{-0.3mm}$ if and only if ${s({N}_{\infty}+{A'}_{\!\infty})-(1+\sum\hspace{-0.1mm}{a_k}){N}_{\infty}-1} \geq 0$.

\item {\sl Case 3:} $N - \sum\hspace{-0.25mm}A_i = 0$. \quad In this case, 
we put $y_\infty = x_2/x_1$ and $z=x_0/x_1$.
One can confirm that $\pi^{*}(\omega_{s,\bm{a}})$ is given as
\begin{equation}
	\pi^{*}\omega_{(s,\bm{a})} = \frac{-z^{s-\sum\hspace{-0.25mm}a_i-2}\prod(1-\lambda_iz)^{a_i}}{\y^s}dz
\end{equation}
clearly, so $\pi^{*}(\omega_{s,\bm{a}})$ is regular at every
$Q_{\infty} \in \pi^{-1}(\{(0:1:\zeta) ; \zeta^N=1\})$ 
if and only if $s-\sum\hspace{-0.25mm}a_i-2 \geq 0$.
\end{itemize}
Thus the proposition holds in every case.
\end{proof}
\end{Prop}
Using the discussion above, we describe the space $\varOmega[X]$.
\begin{Thm}\label{TheoremSection3}
Assume that $K$ contains $\mu_N$ and $\{\lambda_0,\ldots,\lambda_r\}$. For $0 \leq s \leq N-1$, let $V_s$ be the subspace of $\varOmega[X]$ having the character $\zeta \mapsto \zeta^s$ under the action $(x,y)\to (x,\zeta y)$ of $\mu_n$ on $\varOmega[X]$. Note $\varOmega[X]= \bigoplus_{s=0}^{N-1} V_s$.
Put
\[
	\begin{aligned}
	d_s &= \max\,\biggl\{0,\left\lfloor \frac{s\sum{A_k} - (N,N-\sum{A_k})}{N} \right\rfloor -\sum_{j=0}^r \left\lceil \frac{sA_j + (N,A_j)}{N}-1\right\rceil\biggr\},\\
	e_{s,j} &= \left\lceil \frac{sA_j + (N,A_j)}{N}-1 \right\rceil\text{ with }\bm{e_s}=(e_{s,0},e_{s,1},\ldots,e_{s,r}).
	\end{aligned}
\]
We have $\dim V_s=d_s$. Moreover, a basis of $V_s$ is given by
\[
	x^m \omega_{(s,\bm{e_s})} = x^m \frac{\prod_{j=0}^r(x-\lambda_j)^{e_{s,j}}}{y^s}dx
\]
for $0 \leq m \leq d_s-1$.
\begin{proof}
It is obvious that $x^m \omega_{(s,\bm{e_s})} \in V_s$ for all $0 \leq m \leq d_s-1$, since $x^m$ is a linear combination of $(x-\lambda_0)^k$ for $0\le k\le m$ and for $\bm{a} = \bm{e_s} + (k,0,\ldots,0)$, so we have $\omega_{(s,\bm{a})} \in \varOmega[X]$ by Proposition 3.4 and Proposition 3.5. For the converse, an arbitrary element of $V_s$ is a linear combination of $\omega_{(s, \bm{a})}$ with $a_j \geq e_{s,j}$ for all $j \in \{0,\ldots,r\}$ and
\[
	\sum_{k=0}^r a_k \leq \left\lfloor \frac{s\sum{A_k} - (N,N-\sum{A_k})}{N} \right\rfloor - 1.
\]
Now rewriting $\omega_{(s, \bm{a})} = \varphi(x)\omega_{(s,\bm{e_s})}$ where $\deg\varphi \leq d_s-1$, then $\varphi(x)$ is spanned by $\{1,x,\ldots,x^{d_s-1}\}$. Clearly these are linearly independent.
\end{proof}
\end{Thm}
\begin{Rmk}
If $K$ is perfect and contains $\mu_N$, then each member of the basis obtained above is defined over $K$.
In fact, since $\omega_{(s,\bm{e_s})}$ is the unique ``monic" element having the lowest-degree, it is stable under the action of ${\rm Gal}(\mbar{K}/K)$.
\end{Rmk}
\section{Space of regular differential forms on $C$}
\setcounter{equation}{0}
In this section, we consider regular differential forms on $C$.
As we see in Lemma 2.5, $C$ has singularities. We refer to \cite[Chapter\,IV, $\S 3.9$]{Serre} for the regular differential forms of singular curves. Let us give a brief review of it. We set $\varOmega[C] := \bigcap_{P \in C} \varOmega[C]_{P}$ with 
\begin{equation}\label{SerreRegDiffForm}
	\varOmega[C]_{P} := \biggl\{\omega \in \varOmega(X); \raise1ex\hbox{$\displaystyle\sum_{\pi(Q) = P}$}{\rm res}_{Q}\bigl(\pi^{*}(h)\omega\bigr)=0, \text{ for all } h \in {\mathcal O}_{C,P}\biggr\},
\end{equation}
where $\pi: X \rightarrow C$ is the desingularization map we constructed in Section 2. Note that $\varOmega[C]_{P}$ is a ${\mathcal O}_{C,P}$-module and
that $\varOmega[C]$ is the space of global sections of the sheaf $U \mapsto \bigcap_{P \in U} \varOmega[C]_{P}$, which is the dualizing sheaf on $C$.

First, we look at $\varOmega[C]_{P_j}$ for $j=0,\ldots,r$.
We use the notation of Proposition \ref{DesingFinitePlace}.
\begin{Lem}\label{lem:4.1}
For $j \in \{0,\ldots,r\}$ we set $d \geq 0$ and $e \in \{1,\ldots,g_j\}$. A differential form $u^{g_j-e}dz/z^{d+1}$ belongs to $\varOmega[C]_{P_j}$ if and only if any pair $(a,b)$ of non-negative integers does not satisfy
\begin{numcases}{}
aN_j + b{A'}_{\!j} = d, \notag\\
-am_j + bn_j \equiv e~({\rm mod}\,g_j).\notag
\end{numcases}
\begin{proof}
Any element of ${\mathcal O}_{C,P_j}$ is 
written as $\alpha h$ for $\alpha\in {\mathcal O}_{C,P_j}^\times$
and for $h := (x-\lambda_j)^ay^b$ for non-negative integers $a,b$.
As $\pi^*(x-\lambda_j) = u^{-m_j}z^{N_j}$ and $\pi^*(y)=u^{n_j}z^{A'_j}$, we have
\[
\pi^*(h)u^{g_j-e}dz/z^{d+1}= u^{-am_j+bn_j+g_j-e}z^{aN_j+bA'_j} dz/z^{d+1}.
\]
The sum of the residues at
$Q_j = (\lambda_j,u,0) \in \pi^{-1}(\{P_j\})$ where $u$ runs among $g_j$-th roots of $f_j(\lambda_j)$ is not zero if and only if
$aN_j+bA'_j = d$ and $-am_j + bn_j \equiv e \pmod{g_j}$.
\end{proof}
\end{Lem}

According to Lemma \ref{lem:4.1} and Proposition \ref{B4}, we directly obtain
\begin{enumerate}
\setlength{\itemsep}{0cm}
\renewcommand{\labelenumi}{(\roman{enumi})}
\item[(i)] if $d \geq g_j{N}_{j}{A'}_{\!j}-{N}_{j}-{A'}_{\!j} + 1$, then $u^{g_j-e}dz/z^{d+1} \hspace{-0.3mm}\notin \varOmega[C]_{P_j}$ for all $e \in \{1,\ldots,g_j\}$;
\item[(ii)]  for $d_0 = g_j{N}_{j}{A'}_{\!j}-{N}_{j}-{A'}_{\!j}$ and $e_0 \equiv m_j-n_j~({\rm mod}\ g_j)$,
we have $u^{g_j-e_0}dz/z^{d_0+1} \in \varOmega[C]_{P_j}$.
Since $u^{g_j}=f_j(x)\in {\mathcal O^\times_{C,P_j}}$, we do not need to care about the choice of $e_0$.
\end{enumerate}
Moreover,
\begin{enumerate}
\item[(iii)]
the differential form $u^{g_j-e_0}dz/z^{d_0+1}$ in (ii) is a generator of $\varOmega[C]_{P_j}$.
Indeed, by Lemma \ref{lem:4.1} and Proposition \ref{B:Duality},
any $u^{g_j-e}dz/z^{d+1} \in \varOmega[C]_{P_j}$
can be written as $(x-\lambda_j)^{a'}y^{b'}u^{g_j-e_0}dz/z^{d_0+1}$
for non-negative intgers $a'$ and $b'$,
up to multiple of an element of ${\mathcal O^\times_{C,P_j}}$.
\end{enumerate}

Let us rewrite the generator of $\varOmega[C]_{P_j}$.
\begin{Lem}\label{Lem:C-FinitePlace}
For $j \in \{0,\ldots,r\}$, the pull-back $\pi^*(dx/y^{N-1})$ is a generator of $\varOmega[C]_{P_j}$.
\end{Lem}
\begin{proof}
By \eqref{Eq:ProofOfProp3.4}, up to multiple of an element of ${\mathcal O}_{C,P_j}^\times$,
the pull-back $\pi^*(dx/y^{N-1})=\pi^{*}\omega_{(N-1, \bm{0})}$ is equal to
\[
u^{g_j-(N-1)n_j-m_j}z^{-(N-1)A'_j + N_j -1 }dz.
\]
This is of the same form as in (ii) above by $N=g_jN_j$.
\end{proof}

Next, we describe a generator of $\varOmega[C]_{\Pinf}$.
We use the notation of Proposition \ref{DesingInfinitePlace}.

In {\sl Case 1}, Lemma \ref{lem:4.1} holds after replacing $j$ by $\infty$.
Hence we similarly obtain:
\begin{enumerate}
\setlength{\itemsep}{0cm}
\renewcommand{\labelenumi}{(\roman{enumi})}
\item if $d \geq g_{\infty}{N}_{\infty}{A'}_{\!\infty}-{N}_{\infty}-{A'}_{\!\infty} + 1$, then $u^{g_{\infty}-e}dz/z^{d+1} \hspace{-0.3mm}\notin \varOmega[C]_{\Pinf}$ for all $e \in \{1,\ldots,g_{\infty}\}$;
\item if $d_0 = g_{\infty}{N}_{\infty}{A'}_{\!\infty}-{N}_{\infty}-{A'}_{\!\infty},\,e_0 \equiv m_{\infty}-n_{\infty}~({\rm mod}\ g_{\infty})$, then $u^{g_{\infty}-e_0}dz/z^{d_0+1} \in \varOmega[C]_{\Pinf}$;
\item[(iii)] the differential form $u^{g_{\infty}-e_0}dz/z^{d_0+1}$ in (ii) is a generator of $\varOmega[C]_{P_\infty}$.
\end{enumerate}
\begin{Lem}
In {\sl Case 1}, the pull-back $\pi^{*}\omega_{(N-1,\bm{a})}$ is a generator of $\varOmega[C]_{\Pinf}$ if $\sum a_k = N-3$. 
\begin{proof}
By \eqref{Eq:3.2}, up to multiple of an element of ${\mathcal O}^\times_{C,P_\infty}$,
the pull-back $\pi^*\omega_{(N-1,\bm{a})}$ with $\sum {a_k} = N-3$
is equal to
\[
u^{g_\infty - (N-1)n_\infty - m_\infty}z^{(N-1)A'_\infty+N_\infty-1}dz.
\]
This is of the same form as in (ii) above by $N=g_\infty N_\infty$.
\end{proof}
\end{Lem}

In {\sl Case 2}, we set $N'_\infty = {N}_{\infty}+{A'}_{\!\infty}$.
Then Lemma \ref{lem:4.1} holds after replacing $N_j$ by $N'_\infty$
and $j$ by $\infty$.
We similarly obtain:
\begin{enumerate}
\setlength{\itemsep}{0cm}
\renewcommand{\labelenumi}{(\roman{enumi})}
\item if $d \geq g_{\infty}{N'}_{\hspace{-1mm}\infty}{A'}_{\!\infty}-{N'}_{\hspace{-1mm}\infty}-{A'}_{\!\infty} + 1$, then $w^{g_{\infty}-e}dz/z^{d+1} \hspace{-0.3mm}\notin \varOmega[C]_{\Pinf}$ for all $e \in \{1,\ldots,g_{\infty}\}$;
\item if $d_0 = g_{\infty}{N'}_\infty {A'}_{\!\infty}-{N'}_{\hspace{-1mm}\infty}-{A'}_{\!\infty},\,e_0=2n_{\infty}-m_{\infty}~({\rm mod}\ g_{\infty})$, then $w^{g_{\infty}-e_0}dz/z^{d_0+1} \hspace{-0.3mm}\in \varOmega[C]_{\Pinf}$;
\item the differential form $w^{g_{\infty}-e_0}dz/z^{d_0+1}$ in (ii) is a generator of $\varOmega[C]_{P_\infty}$.
\end{enumerate}

\begin{Lem}
In {\sl Case 2}, the pull-back $\pi^{*}\omega_{(2-A_{\sinfty},\bm{a})}$ is a generator of $\varOmega[C]_{\Pinf}$
if $\sum{a_k} = 0$.
\begin{proof}
By \eqref{Eq:3.3}, up to multiple of an element of ${\mathcal O}^\times_{C,P_\infty}$,
the pull-back $\pi^*\omega_{(2-A_{\sinfty},\bm{a})}$ with $\sum {a_k} = 0$
is equal to
\[
w^{g_\infty - (2-A_\infty)n_\infty + (1-A_\infty)m_\infty}
z^{(2-A_\infty)(N_\infty + A'_\infty) - N_\infty -1}dz.
\]
This is of the same form as in (ii) above by $A_\infty=g_\infty A'_\infty$
and $N'_\infty = N_\infty + A'_\infty$.
\end{proof}
\end{Lem}
We obtain the regularity of some rational differential forms on $C$,
which will turn out to make a basis of the space of
regular differential forms on $C$ (cf.\ Corollary \ref{CorollarySection4}).
\begin{Thm}\label{Theorem-C}
\begin{enumerate}
\item Assume $N - \sum{A_k} \geq 0$. Then, for $0\leq s \le N-1$,
we have $\pi^{*}\omega_{(s,\bm{a})} \in \varOmega[C]$  if
\begin{enumerate}
\item[(i)] $a_j \geq 0$ for all $j \in \{0,\ldots, r\}$ and
\item[(ii)] $0 \leq \sum{a_k} \leq s-2$.
\end{enumerate}
\item Assume $N - \sum{A_k} < 0$.
Then, for $2-A_\infty\le s \le N-1$, we have
$\pi^{*}\omega_{(s,\bm{a})} \in \varOmega[C]$ if
\begin{enumerate}
\item[(i)] $a_j \geq 0$ for all $j \in \{0,\ldots, r\}$ and
\item[(ii)] $0 \leq \sum {a_k} \leq s-2+A_{\infty}$.
\end{enumerate}
\end{enumerate}
\begin{proof}
A differential form $dx/y^{N-1}$ and its products of some $x, y$ are regular at $P_j$ for $j \in \{0,\ldots,r\}$ by Lemma 4.2. 
(1)
In {\sl Case 1}, the differential form $\omega_{(N-1,\bm{a}')}$ for $\sum{a'_k} = N-3$ and its products of some $1/x, y/x$ are regular at $P_{\infty}$ by Lemma 4.3. The theorem in this case follows from the fact that
$\pi^{*}\omega_{(s,\bm{a})}$ for $\bm{a}$ satisfying (i) and (ii)
is a linear combination of
\[
\left(\frac{1}{x}\right)^i\left(\frac{y}{x}\right)^{N-1-s}\frac{x^{N-3}dx}{y^{N-1}}
\]
for $0\le i\le s-2$.
In {\sl Case 3}, recall that $P_{\infty}$ are nonsingular points by Lemma 2.5. Then $\varOmega[C]_{\Pinf}$ is equal to the set of differential forms which is regular at $Q_{\infty} \in \pi^{-1}(\{P_{\infty}\})$. So this case follows from (3.4).

(2) In {\sl Case 2}, then a differential form $\omega_{(2-A_{\sinfty}, \bm{a}')}$ where $\sum a'_k = 0$ and its products of some $1/y, x/y$ are regular at $P_{\infty}$ by Lemma 4.4. The theorem in this case follows from the fact that
$\pi^{*}\omega_{(s,\bm{a})}$ for $\bm{a}$ satisfying (i) and (ii)
is a linear combination of
\[
\left(\frac{1}{y}\right)^i\left(\frac{x}{y}\right)^{s-2+A_\infty-i} \frac{dx}{y^{2-A_\infty}}
\]
for $0 \le i \le s-2+A_\infty$.
\end{proof}
\end{Thm}

As an application, we obtain a basis of the regular differential module on $C$. For each $\zeta \in \mu_N$, we have the automorphism of $C$ defined by $(x,y) \mapsto (x, \zeta y)$, which can be extended to an automorphism of $X$, say $\iota_\zeta$. This induces an action of $\mu_N$ on $\varOmega[C]$. Indeed $\iota_\zeta$ stabilizes $P_j$ and induces a permutation of $\pi^{-1}(\{P_j\})$ and we have ${\rm res}_{Q}(\omega) = {\rm res}_{\iota_\zeta(Q)}({\iota_\zeta}^*\omega)$ for each $Q \in \pi^{-1}(\{P_j\})$, thanks to Remark\,\ref{RemarkOnDesing}.

\begin{Cor}\label{CorollarySection4}
Assume that $K$ contains $\mu_N$ and $\{\lambda_0,\ldots,\lambda_r\}$. For $0 \leq s \leq N-1$, let $W_s$ be the subspace of $\varOmega[C]$ consisting of $\omega\in \varOmega[C]$ on which $\mu_N$ acts by $\omega \mapsto \zeta^s\omega$ for all $\zeta\in \mu_N$.
\begin{enumerate}
\item If $N - \sum{A_k} \geq 0$, then
\begin{equation}\label{Basis_C_first}
	\{\pi^*(x^idx/y^s)\,;\,0\le i \le s-2\}
\end{equation}
is a basis of $W_s$. In particular
\begin{equation}
	\dim W_s = s-1
\end{equation}
with $\dim \varOmega[C] = (N-1)(N-2)/2$.
\item If $N - \sum{A_k} < 0$, then 
\begin{equation}\label{Basis_C_second}
	\biggl\{\pi^*(x^iy^{(j-1)N}dx/y^s)\,;\,0 \leq i \leq 
s-2-jN+\sum{A_k},\ 1 \leq j \leq \left\lfloor\frac{s-2+\sum A_k}{N}\right\rfloor\biggr\}
\end{equation}
is a basis of $W_s$. In particular\vspace{-1mm}
\begin{equation}\label{DimensionFormula_C_second}
	\dim W_s = \sum_{j=1}^{\left\lfloor \frac{s-2+\sum\!A_k}{N}\right\rfloor}{\left(s-1-jN+\sum{A_k}\right)}
\end{equation}
with $\dim \varOmega[C] = (-1+\sum A_k)(-2+\sum A_k)/2$.
\end{enumerate}
\begin{proof}
(1) The differential forms in \eqref{Basis_C_first} belong to $W_s$ by Theorem 4.5\,(1), and they are linear independent. By using \cite[Corollary\,III.9.10]{Hartshorne}, the arithmetic genus does not change among fibers of a flat family over a connected noetherian scheme, and therefore so does the dimension of the space of global sections of the dualizing sheaf by definition \cite[Chapter\,III, \S 7]{Hartshorne}. Consider the family
\[
	y^N = \prod_{j=0}^r \prod_{k=1}^{A_j} (x-\gamma_{jk}z) \cdot \prod_{l=1}^{A_{\infty}}(-a_lx+z)
\]
which has $C$ as a special fiber (defined by $\gamma_{jk}=\lambda_j$ and $a_l=0$). Since its generic fiber has arithmetic genus $(N-1)(N-2)/2 = \sum(s-1)$. This shows that the differential forms in \eqref{Basis_C_first} span $W_s$.

(2) The differential forms in \eqref{Basis_C_second} belong to $W_s$ by Theorem 4.5\,(2), and they are linear independent. In the same way as in (1), using the fact that the projective model of $C$:
\[
	y^Nz^{A_{\infty}} = \prod_{j=1}^{r} (x-\lambda_j z)^{A_j}
\]
deforms to a smooth curve of degree $\sum{A_k}$, we have $\dim \varOmega[C]=(-1+\sum A_k)(-2+\sum A_k)/2$. Hence, to prove \eqref{DimensionFormula_C_second}, it suffices to show that the sum of the right hand side of \eqref{DimensionFormula_C_second} for $s=0,\ldots,N-1$ is equal to $(-1+\sum A_k)(-2+\sum A_k)/2$.
This follows from the next fact. The set
\[
	\biggl\{ s-1-jN+\sum A_k\,;\,0\le s \le N-1, 1\leq j \leq \left\lfloor \frac{s-2+\sum A_k}{N}\right\rfloor \biggr\}
\]
is equal to $\{1,\ldots, -2 + \sum A_k\}$, since $1\leq s-1-jN+\sum A_k \leq -2 + \sum A_k$.
\end{proof}
\end{Cor}
\begin{Exp}\label{Example4.7}
Let $C: y^3 = x(x-1)^2(x-z)^2$ and $X$ be the desingularizaiton of $C$.
Then
\[
	\Biggl\{\pi^*(dx),\pi^*\biggl(\frac{dx}{y}\biggr),\pi^*\biggl(\frac{xdx}{y}\biggr),\pi^*\biggl(\frac{dx}{y^2}\biggr),\pi^*\biggl(\frac{xdx}{y^2}\biggr),\pi^*\biggl(\frac{x^2dx}{y^2}\biggr)\Biggr\}
\]
is a basis of $\varOmega[C]$, while $\pi^*(dx/y),\pi^*(x^2dx/y^2) \in \varOmega[X]$.
\end{Exp}
\section{The space of regular differential forms on $\widetilde{C}$}
\setcounter{equation}{0}
One can consider partial desingularizations of $C$, i.e.,
the desingularaization only around a subset of $\{P_0,\ldots, P_r,P_\infty\}$.
Avoiding the general setting,
in this section we consider the desingularization $\widetilde{C}$ only around $P_{\infty}$ of the curve $C$, which would be the most interesting case.
We can exclude {\sl Case 3}, since ${\widetilde C} = C$ in {\sl Case 3}.
Our aim is to give an explicit basis of the space of the regular differential forms on $\widetilde{C}$. Now $\widetilde{C}$ is described as below:
\begin{Def}
Let $C$ be a curve associated to Appell-Lauricella hypergeometric series.
We define $\widetilde{C}$ by gluing $X_\infty$
and $C \smallsetminus \{P_\infty\}$.
\end{Def}
Next we construct the morphism $\widetilde{\pi}: X \rightarrow \widetilde{C}$ as follows. Recall that $X$ was defined by gluing $X_i$ in Section 2, then it suffices to define the morphism $\widetilde{\pi}_i: X_i \rightarrow \widetilde{C}$ for each $i \in \{0,\ldots,r,\infty\}$.
We define $\widetilde{\pi}_\infty$ to be the inclusion $X_\infty \subset \widetilde{C}$
and $\widetilde{\pi}_i$ for $i=0,\ldots,r$
to be the composition of $\pi_j: X_i\to C \smallsetminus \{P_0,\overset{\overset{i}{\vee}}{\ldots},P_r, P_\infty\}$ and the inclusion $C \smallsetminus \{P_0,\overset{\overset{i}{\vee}}{\ldots},P_r, P_\infty\} \subset C \smallsetminus \{P_\infty\}$.
Here is a description of the regularity of differential forms and an explicit basis on $\widetilde{C}$.
\begin{Thm}\label{TheoremSection5}
We have 
$\widetilde{\pi}^{*}\omega_{(s,\bm{a})} \in \varOmega[\widetilde{C}]$
if
\begin{enumerate}
\item[(i)] $a_j \geq 0$  for all $j \in \{0,\ldots, r\}$ and
\item[(ii)]  $0 \leq \sum{a_k} \leq \displaystyle\frac{s\sum{A_k} - (N, N-\sum{A_k})}{N}-1$.
\end{enumerate}
\begin{proof}
The fiber $Q_{\infty} \hspace{-0.3mm}\in \pi^{-1}(\{P_{\infty}\})$ are not singular points of $\widetilde{C}$, so $\widetilde{\pi}^{*}\omega_{(s,\bm{a})}$ is an element of $\varOmega[\widetilde{C}]_{\Qinf}$ if and only if $\widetilde{\pi}^{*}\omega_{(s,\bm{a})} \in \varOmega(X)$ is regular at $Q_{\infty}$. Hence, use Proposition \ref{Prop_X_infinite} and Lemma \ref{Lem:C-FinitePlace}.
\end{proof}
\end{Thm}

Assume that $K$ contains $\mu_N$.
As with the case of $\varOmega[C]$ in Section 4, the automorphism $(x,y) \mapsto (x, \zeta y)$ for $\zeta \in \mu_N$ induces an action of $\mu_N$ on $\varOmega[\widetilde{C}]$. Let $\widetilde{W}_s$ be the subspace of $\varOmega[\widetilde{C}]$ consisting of $\omega\in \varOmega[\widetilde{C}]$ on which $\mu_N$ acts by $\omega \mapsto \zeta^s\omega$ for all $\zeta\in \mu_N$.
\begin{Cor}\label{cor:basis_for_tilde_C}
For $0\le s \le N-1$, the set of $\widetilde\pi^*\omega_{s,j}$ with
\begin{equation}\label{DiffFormOfCor5.3}
	\omega_{s,j} := \frac{x^{j-1}dx}{y^s},\quad 1 \leq j \leq \frac{s\sum{A_k} - (N,N-\sum{A_k})}{N}
\end{equation}
is a basis of $\widetilde{W}_s$. In particular
\[
	\dim \widetilde{W}_s = \max\left\{0, \left\lfloor\frac{s\sum{A_k} - (N,N-\sum{A_k})}{N}\right\rfloor\right\}.
\]
\begin{proof}
By Theorem \ref{TheoremSection5}, the differential forms
$\omega_{s,j}$ in \eqref{DiffFormOfCor5.3} belong to $\widetilde{W}_s$ and they are linearly independent obviously. Let $n=\sum A_k$ and $\mathcal{H}$ be the family
\[
	y^N = \prod_{j=0}^r \prod_{k=1}^{A_j} (x-\gamma_{jk})
\]
where $(\gamma_{jk}) \in \mathbb{A}^n$. Let $\widetilde{\mathcal{H}}$ be the familiy over $\mathbb{A}^n$ obtained as the fiberwise desingularization only at $\infty$ of $\mathcal{H}$. We remark that $\varOmega[\widetilde{C}]$ is the space of global secions of the dualizing sheaf of $\widetilde{C}$. By the similar way as in Corollary \ref{CorollarySection4}, the dimension of $\varOmega[\widetilde{\mathcal{H}}]$ is equal to $\dim\varOmega[{\widetilde{{\mathcal{H}}}}(t)]$ for a smooth fiber $\widetilde{\mathcal{H}}(t)$ with $t \in \mathbb{A}^n$ of $\widetilde{\mathcal{H}}$. And $\dim\varOmega[{\widetilde{{\mathcal{H}}}}(t)]$ is equal to the dimension we have computed in Theorem \ref{TheoremSection3} for $A_j=1$ for $j=1,\ldots,r$, which is 
\[
	\sum_{s=0}^{N-1} \max\left\{0, \left\lfloor\frac{s\sum{A_k} - (N,N-\sum{A_k})}{N}\right\rfloor\right\}.
\]
This implies that $\{\omega_{s,i}\}$ has to span $\widetilde{W}_s$.
\end{proof}
\end{Cor}
\begin{Exp}\label{Example5.4}
We consider the curve $C: y^3 = x(x-1)^2(x-z)^2$ in Example \ref{Example4.7}, then
\[
	\Biggl\{\pi^*\biggl(\frac{dx}{y}\biggr),\pi^*\biggl(\frac{dx}{y^2}\biggr),\pi^*\biggl(\frac{xdx}{y^2}\biggr),\pi^*\biggl(\frac{x^2dx}{y^2}\biggr)\Biggr\}
\]
is a basis of $\varOmega[\widetilde{C}]$.
\end{Exp}
\section{The modified Cartier operator on the regular differential module}
\setcounter{equation}{0}
In this section, we assume that $K$ is a perfect field of characteristic $p > 0$. Let $C$ be the projective model of
\[
y^N=f(x) = x^{A_0}(x-\lambda_1)^{A_1}\cdots (x-\lambda_r)^{A_r},
\]
i.e., a curve associated to Appell-Lauricella hypergeometric series as in Definition 2.1. We introduce the (modified) Cartier operator
 on the regular differential modules on $X$, $C$ and $\widetilde C$, studied in the previous sections
and describe it in terms of Appell-Lauricella hypergerometric series,
where $X$ is the desingularization of $C$ and $\widetilde{C}$ is the partial desingularization of $C$ only at $\infty$ (Definition 5.1).

We start with recalling \cite[Definition 2.1]{Yui} the definition of the modified Cartied operator $\mathcal{C}'$ on the space $\varOmega(C)$
of rational differential forms on $C$.
\begin{Def}
For all $\omega \in \varOmega(C)$, there exist $\phi, \eta \in K(x,y)$ such that $\eta^p \in K(x^p,y^p)$ and
\[
	\omega = d\phi + \eta^px^{p-1}dx.
\]
Then, the modified Cartier operator $\mathcal{C}': \varOmega(C) \longrightarrow \varOmega(C)$ is defined as $\mathcal{C}'(\omega) = \eta dx$.
\end{Def}

It is well-known that the modified Cartier operator $\mathcal{C}'$ stabilizes
$\varOmega[X]$. The next theorem says that this also holds for
$\varOmega[C]$ and $\varOmega[\widetilde{C}]$.
\begin{Thm}\label{CartierOperatorOnC}
The modified Cartier operator $\mathcal{C}'$ stabilizes $\varOmega[C]$ and $\varOmega[\tilde C]$.
More generally, $\mathcal{C}'$ stabilizes $\varOmega[C]_P$ for any closed point $P\in C$,
where $\varOmega[C]_P$ is as defined in \eqref{SerreRegDiffForm}.
\end{Thm}
\begin{proof}
It suffices to show the second assertion.
Let $\omega\in \varOmega[C]_P$. Write
\begin{equation}\label{writing_omega}
\omega = d\phi + \eta^p x^{p-1} dx.
\end{equation}
Let $h$ be an arbitrary element of ${\mathcal O}_{C,P}$. 
Mutiplying \eqref{writing_omega} by $\pi^*(h^p)$, we get
\[
\pi^*(h^p) \omega = d(\pi^*(h^p)\phi) + (\pi^*(h)\eta)^p x^{p-1} dx.
\]
We have
\[
\left(\sum_{\pi(Q)=P} {\rm res}_Q(\pi^*(h)\eta dx)\right)^p
= \sum_{\pi(Q)=P} {\rm res}_Q((\pi^*(h)\eta)^p x^{p-1} dx)
= \sum_{\pi(Q)=P} {\rm res}_Q(\pi^*(h^p) \omega).
\]
The right hand side is zero by $\omega \in \varOmega[C]_P$.
Hence, $\sum_{\pi(Q)=P} {\rm res}_Q(\pi^*(h)\eta dx) = 0$ and therefore
$\mathcal{C}'(\omega) = \eta dx \in \varOmega[C]_P$.
\end{proof}

\begin{Def}
{\it The Cartier-Manin matrix} $A$ of $X$ (resp. $C$ or $\widetilde C$) with respect to a basis $\{\xi_i\}$ of
$\varOmega[X]$ (resp. $\varOmega[C]$ or $\varOmega[\widetilde C]$)
is given by $A=(a_{ij})$ with
${\mathcal C}' \xi_j = \sum_{i} a_{ij}^{1/p} \xi_i$.
\end{Def}

In order to describe a Cartier matrix (i.e, the modified Cartier operator $\mathcal{C}'$ on $\varOmega[X]$, $\varOmega[C]$ or $\varOmega[\tilde C]$), it suffices to describe ${\mathcal C}'$ on 
\[
\omega_{s,j} = \frac{x^{j-1}}{y^s}\ dx
\]
for $0\le s \le N-1$,
since our basis of the space of regular differential forms (obtained in the previous sections) is given by
linear combinations of $\omega_{s,j}$.
We shall see that it can be described in terms of the Appell-Lauricella hypergerometric series.
Recall $p\nmid N$.
Hence, for $1 \le s \le N-1$ there uniquely exist integers $m'_s$ and $n'_s$ with
$1 \le m'_s \le N-1$ and $0\le n'_s < p$ such that
\begin{equation}
{m'_s}p-{n'_s}N = s
\end{equation}
by Lemma B.5.
We rewrite $\omega_{s,j}$ as
\[
	\omega_{s,j} = y^{-s}x^{j-1}dx = y^{-m'_s p}x^{j-1}y^{m'_s p-s}dx = {(y^{m'_s})}^{-p}x^{j-1}{f(x)}^{n'_s}dx.
\]
Let $\gamma_{s,e}$ be the coefficient of $x^e$ in the polynomial $f(x)^{n'_s}$, namely
\begin{equation}
	{f(x)}^{n'_s} = \sum_{e=0}^{n'_s\deg(f)} \gamma_{s,e} x^e.
\end{equation}
Now we have
\begin{align}\label{FundamentalEquationForCartierOperator}
	\omega_{s,j} &= (y^{m'_s})^{-p}\hspace{-2mm}\hbox{$\displaystyle\sum_{j+e \not\equiv 0\,({\rm mod}\,p)}\hspace{-6mm}\gamma_{s,e}x^{j+e-1}dx$} + \sum_{l} \gamma_{s,(l+1)p-j} \frac{x^{(l+1)p}}{y^{m'_s p}}\frac{dx}{x}\notag\\
	&= d\biggl(y^{-m'_s p}\raise1ex\hbox{$\displaystyle\sum_{j+e \not\equiv 0\,({\rm mod}\,p)}$}\frac{\gamma_{s,e}x^{j+e}}{j+e}\biggr) + \sum_{l} \gamma_{s,(l+1)p-j} \frac{x^{lp}}{y^{m'_s p}}x^{p-1}dx
\end{align}
where $l$ runs from $\displaystyle \left\lceil\frac{j}{p}-1\right\rceil$ to $\displaystyle \left\lfloor\frac{n'_s\deg(f)+j}{p}-1\right\rfloor$.

Let ${\mathcal F}(a,b_2,\ldots,b_r,c\,;\lambda_2,\ldots,\lambda_r)$ be
the Appell-Lauricella hypergeometric series associated to $C: y^N=x^{A_0}(x-1)^{A_1}\prod_{k=2}^r (x-\lambda_k)^{A_k}$, namely
\[
	a = -1 + \sum_{k=0}^r(A_k/N), \quad b_i=A_i/N, \quad c= a + 1-(A_1/N).\vspace{-1mm}
\]

As well as in the case of elliptic curves, we need to 
introduce a truncation of hypergeometric series.

\begin{Def}[Truncation of Appell-Lauricella hypergeometric series]\label{def:truncation}
For $(\sigma, \tau_1,\ldots,\tau_r)\in {\mathbb Z}^{r+1}$, let
\[
{\mathcal F}^{(\sigma; \tau_1,\ldots,\tau_r)}(a,b_2,\ldots,b_r,c\,;\lambda_2,\ldots,\lambda_r)
\]
be the sum of the $\lambda_2^{e_2}\cdots\lambda_r^{e_r}$-terms of ${\mathcal F}(a,b_2,\ldots,b_r,c\,;\lambda_2,\ldots,\lambda_r)$ 
for $(e_2,\ldots,e_r)$ satisfying
$e_j \le \tau_j$ for $j=2,\ldots, r$ and 
\[\sigma-\tau_1 \le \sum_{k=2}^r e_k \le \sigma.
\]
We call it {\it the truncation of  ${\mathcal F}(a,b_2,\ldots,b_r,c\,;\lambda_2,\ldots,\lambda_r)$ with respect to $(\sigma; \tau_1,\ldots,\tau_r)$.}
\end{Def}

Unfortunately, it is not true in general that
one can describe ${\mathcal C}'$ in terms of
the Appell-Lauricella hypergeometric series
${\mathcal F}(a,b_2,\ldots,b_r,c\,;\lambda_2,\ldots,\lambda_r)$
itself of $C$.
We shall see that ${\mathcal C}'$
can be described in terms of Appell-Lauricella hypergeometric series
associated to
a deformation of $f(x)$
which is separable except for the factor of $x$,
see $f_0(x)$ in \eqref{f_0} below for the explicit form.
The description is as follows.
\begin{Thm}\label{MainTheoremGeneralA_k} 
Let $a'$ be a positive rational number with $a'\equiv s\deg(f)/N-j \pmod p$ and
set $c'=a'+1-s/N$ and $d'=n'_s \deg(f) - (l+1)p + j$. For $\left\lceil \frac{j}{p} -1 \right\rceil \leq l \leq \left\lfloor \frac{n'_s \deg(f) + j}{p} -1\right\rfloor$, 
we have
\[
\gamma_{s,(l+1)p-j} =
\frac{(c'; d')}{(a'; d')}
{\mathcal F}^{(d';n'_s,\ldots,n'_s)}(a',\underbrace{s/N,\ldots,s/N}_{-1+\sum_{k\ge 1} A_k},c'; 
\underbrace{1,\ldots,1}_{A_1-1},\underbrace{\lambda_2,\ldots,\lambda_2}_{A_2},\ldots,\underbrace{\lambda_r,\ldots,\lambda_r}_{A_r}),
\]
where the right hand side (a priori belonging to ${\mathbb Q}[\lambda_2,\ldots,\lambda_r]$) is considered as a polynomial over $\mathbb{F}_p$
(Note that the denominator of any coefficient is coprime to $p$).
\end{Thm}
First we see that it is enough to show the case of $A_1 = A_2 = \cdots = A_r=1$.
\begin{proof}[Reduction to the case of $A_k=1$ for $k=1,\ldots,r$]
Consider
\begin{equation}\label{f_0}
f_0(x) = x^{A_0}(x-1)\prod_{t=2}^{A_1} (x-\lambda_{1t})\prod_{k=2}^r \prod_{t=1}^{A_k} (x-\lambda_{kt}).
\end{equation}
Write
\begin{equation}
f_0(x)^{n'_s} = \sum_c (\delta_0)_{s,c} x^c.
\end{equation}
Then
\begin{equation}\label{gamma_vs_delta}
\gamma_{s,c} = (\delta_0)_{s,c}|_{\lambda_{kt}=\lambda_k \text{ for } k=1,\ldots,r \text{ and } t=1,\ldots,A_k, (k,t)\ne (1,1)}
\end{equation}
with $\lambda_1=1$ holds. We shall see in Proposition \ref{MainProposition} below that
$(\delta_0)_{s,(l+1)p-j}$ is described as
\begin{equation}\label{DescriptionOfdelta_0}
\frac{(c'; d')}{(a'; d')}
{\mathcal F}^{(d';n'_s,\ldots,n'_s)}(a',\underbrace{s/N,\ldots,s/N}_{-1+\sum_{k\ge 1} A_k},c'; 
\lambda_{12},\ldots,\lambda_{1A_1},
\lambda_{21},\ldots,\lambda_{2A_2},\ldots,
\lambda_{r1},\ldots,\lambda_{rA_r}),
\end{equation}
which is the result in the case of $A_k=1$ for $k=1,\ldots,r$.
The theorem follows from \eqref{gamma_vs_delta} and \eqref{DescriptionOfdelta_0}.
\end{proof}

\begin{Lem} \label{Binomial_vs_Pochhammer}
Let $a'$ be a positive rational number with $a'\equiv s\deg(f)/N-j \pmod p$ and
set $c'=a'+1-s A_1 /N$ and $d'=n'_s \deg(f) - (l+1)p + j$.
For a partition $(d_1,\ldots,d_r)$ of $d'$ (i.e., $d'=\sum_{k=1}^r d_k$)
with $0\leq d_k < p$, the following is true:
\[
(-1)^{d_k}\binom{n'_s A_k}{d_k} = (-1)^{d_k}\binom{-sA_k/N}{d_k}= \frac{(sA_k/N\,;d_k)}{(1\,;d_k)}
\]
in ${\mathbb F}_p$ for $k = 1,\ldots,r$. Moreover for $k=1$ we have
\[
(-1)^{d_1}\binom{n'_s A_1}{d_1} =
\frac{(c';d')}{(a';d')} \frac{(a';d'-d_1)}{(c';d'-d_1)}
\]
in ${\mathbb F}_p$.
\end{Lem}
\begin{proof}
Since $m'_s p - n'_s N = s$, we have $n'_s = -s/N$ in ${\mathbb F}_p$.
Hence $\binom{n'_s A_k}{d_k} = \binom{-sA_k/N}{d_k}$ for $k=1,\ldots,r$.
The first equality follows from
\begin{eqnarray*}
(-1)^{d_k}\binom{-sA_k/N}{d_k} &=& (-1)^{d_k}\frac{-s A_k/N(-sA_k/N-1)\cdots (-sA_k/N - d_k+1)}{d_k!} \\
&=&\frac{s A_k/N(sA_k/N+1)\cdots (sA_k/N + d_k-1)}{d_k!} = \frac{(sA_k/N\,;d_k)}{(1\,;d_k)}
\end{eqnarray*}
for $k=1,\ldots,r$.

We prove the second equation by induction on $d_1$.
If $d_1 = 0$, then the both sides are equal to one.
Assume that the equation holds for smaller $d_1$. Then
\[
(-1)^{d_1}\binom{n'_s A_1}{d_1} =
 - \frac{n'_sA_1-d_1+1}{d_1}\cdot (-1)^{d_1-1}\binom{n'_s A_1}{d_1-1}
\]
and
\[
\frac{(c';d')}{(a';d')} \frac{(a';d'-d_1)}{(c';d'-d_1)}
= \frac{c'+d'-d_1}{a'+d'-d_1}
\cdot \frac{(c';d')}{(a';d')} \frac{(a';d'-(d_1-1))}{(c';d'-(d_1-1))}.
\]
Then the equality for $d_1$ follows from that
\[
a'+d'-d_1 = (s/N + n'_s)\deg f - d_1 = -d_1\]
in ${\mathbb F}_p$ and that by $c'=a'+1-sA_1/N$ we have
\[c'+d'-d_1 = -d_1+1-sA_1/N = n'_sA_1-d_1+1\]
in ${\mathbb F}_p$.
\end{proof}

\begin{Prop}\label{MainProposition} Assume $A_k = 1$ for $k=1,\ldots,r$.
Let $a'$ be a positive rational number with $a'\equiv s\deg(f)/N-j \pmod p$ and
$c'=a'+1-s/N$. Set $b'_k = s b_k$ and $d'=n'_s \deg(f) - (l+1)p + j$.
Consider a polynomial in $\lambda_2,\ldots,\lambda_r$ over $\mathbb Q$
\[
\delta_{s,(l+1)p-j}:=
\frac{(c';d')}{(a';d')}
{\mathcal F}^{(d';\tau_1,\ldots,\tau_r)}(a',b'_2,\ldots,b'_r,c';\lambda_2,\ldots,\lambda_r),
\]
where $\tau_k:={n'_s}$ for $k=1,\ldots,r$.
\begin{enumerate}
\item[(1)]
The denominator of every coefficient of $\delta_{s,(l+1)p-j}$
is coprime to $p$.
Hence we can consider it as a polynomial over ${\mathbb F}_p$, say $\overline{\delta}_{s,(l+1)p-j}$.
\item[(2)] We have the equality
$\gamma_{s,(l+1)p-j} = \overline{\delta}_{s,(l+1)p-j}$.
\end{enumerate}
\end{Prop}
\begin{Rmk}
For $a'_0= s \deg(f)/N -j$ and $c'_0=s(\deg(f)-1)/N - j + 1$,
the Appell-Lauricella hypergeometric series 
${\mathcal F}(a'_0,b'_2,\ldots,b'_r,c'_0;\lambda_2,\ldots,\lambda_r)$
is associated to
\[
\frac{x^{j-1}}{y^s} dx = \frac{1}{x^{sA_0/N-j+1}(x-1)^{sA_1/N}(x-\lambda_2)^{sA_2/N}\cdots(x-\lambda_r)^{sA_r/N}}.
\]
with $A_k=1$ for $k\ge 1$.
Indeed $a'_0 = -1 + (sA_0/N-j+1) +
\sum_{k=1}^r sA_k/N = s(-1+\sum A_k/N)-j+s = sa-j+s$
and $c'_0=a'_0+1-(sA_1/N)=s(a-A_1/N+1)-j+1=sc-j+1$.
In the theorem, we use a positive $a'$ instead of possibly non-positive $a'_0$ so that $(a';d')$ ($\in{\mathbb Q}$) and the denominators ($\in{\mathbb Q}$) of coefficients of the hypergeometric series are not zero.
\end{Rmk}

\begin{proof}[Proof of Proposition \ref{MainProposition}]
Assume $A_k=1$ for $k=1,\ldots,r$. Then
\[
f(x) = x^{A_0} (x-\lambda_1)(x-\lambda_2)\cdots (x-\lambda_r)
\]
with $\lambda_1 = 1$, and $f(x)^{n'_s}$ is computed as
\begin{eqnarray*}
x^{n'_sA_0} (x-\lambda_1)^{n'_s}(x-\lambda_2)^{n'_s}\cdots (x-\lambda_r)^{n'_s}
&=&x^{n'_sA_0} \prod_{k=1}^r \sum_{d_k=1}^{n'_s} \binom{n'_s}{d_k} (-1)^{d_k}\lambda^{d_k}x^{n'_s-d_k}\\
&=&\sum_{d_1,\ldots,d_k}\prod_{k=1}^r (-1)^{d_k} \binom{n'_s}{d_k} \lambda_1^{d_1}\cdots\lambda_r^{d_r}x^{n'_s \deg(f) - (d_1 + \cdots +d_r)}.
\end{eqnarray*}
The $x^{(l+1)p-j}$-coefficient $\delta_{s,(l+1)p-j}$ of $f(x)^{n'_s}$ is
\[
\sum_{d_1,\ldots,d_k}\prod_{k=1}^r (-1)^{d_k} \binom{n'_s}{d_k} \lambda_1^{d_1}\cdots\lambda_r^{d_r},
\]
where $(d_1,\ldots,d_r)$ runs the set of $(d_1,\ldots, d_r)$ satisfying
$0\le d_k \le n'_s$ and $d_1 + \cdots + d_r = n'_s \deg(f) - (l+1)p - j = d'$.
By Lemma \ref{Binomial_vs_Pochhammer}, this is equal to
\[
\frac{(c' ; d')}{(a' ; d')}\sum_{d_1,\ldots,d_k}\frac{(a',\sum_{k=2}^r d_k)}{(c',\sum_{k=2}^r d_k)}\prod_{k=2}^r
\frac{(s/N ; d_k)}{(1;d_k)}
  \lambda_1^{d_1}\cdots\lambda_r^{d_r},
\]
which is equal to
\[
\frac{(c' ; d')}{(a' ; d')}{\mathcal F}^{(d';n'_s,\ldots,n'_s)}(a',b'_2,\ldots,b'_r, c'; \lambda_2,\ldots,\lambda_r)
\]
by $\lambda_1 = 1$. Thus the proposition was proved.
\end{proof}

\begin{Exp}
We consider the case that $C$ is a (nonsingular) hyperelliptic curve of genus $g \geq 1$. Write $C: y^2 = f(x):=x(x-1)(x-\lambda_2)\cdots (x-\lambda_{2g})$
with separable $f(x)$. Let us describe the Cartier operator on $\varOmega[\widetilde C]$.
The basis is given by $x^{j-1}/y dx$ for $j=1,\ldots,g$ (cf.\ Corollary \ref{cor:basis_for_tilde_C}).
Note that $a=g-1/2$, $b_i=1/2$ and $c=g$.
We get $m'_1 = 1$ and $n'_1 = (p-1)/2$. Put $a'=(2g+1)/2 - j = g+ 1/2 -j$ and $c'=g-j+1$, which are positive,
and set $d' = \frac{p-1}{2}\deg(f) - i p + j$. Then we have
\[
\gamma_{1,ip-j} = \frac{(p(2g-2i+1)-1)!!}{(p(2g-2i+1)-2)!!}
\frac{(2g-2j-1)!!}{(2g-2j)!!}
{\mathcal F}^{(d';n'_1,\ldots,n'_1)}(a',1/2,\ldots,1/2,c';\lambda_2,\ldots,\lambda_{2g})
\]
for $i=1,\ldots,g$. For example, for $(g,p)=(2,3)$ with $f(x)=x(x-1)(x-z_1)(x-z_2)(x-z_3)$, the Cartier-Manin matrix $(\gamma_{1,ip-j})$ is
\[
\begin{pmatrix}
2z_1z_2z_3 + 2z_1z_2 + 2z_1z_3 + 2z_2z_3 & z_1 z_2 z_3\\
1 & 2 z_1 + 2z_2 + 2z_3 + 2
\end{pmatrix}.
\]
Here, ${\mathcal F}(5/2-j,1/2,1/2,1/2,3-j;z_1,z_2,z_3)$ truncated by $z_k$-degree $\le 1$ with coefficients in ${\mathbb Q}$ is
\begin{eqnarray*}
1+ \frac{3}{8}(z_1+z_2+z_3) +
 \frac{5}{32}(z_1z_2 + z_1z_3 + z_2z_3) + \frac{35}{512}z_1 z_2 z_3
& \text{for }j=1,
\\
1+ \frac{1}{4}(z_1+z_2+z_3) +
 \frac{3}{32}(z_1z_2 + z_1z_3 + z_2z_3) + \frac{5}{128}z_1 z_2 z_3
& \text{for }j=2.
\end{eqnarray*}
For further truncations, use $d'=5-3i+j$ and $n'_1=1$ with Definition \ref{def:truncation}.
\end{Exp}

In Theorem \ref{MainTheoremGeneralA_k}, we have shown that
$\varOmega[C]$ and $\varOmega[\widetilde{C}]$ are closed under 
the modified Cartier operator $\mathcal{C}'$.
We can show more as for the subspaces $W_s$ (resp. $\widetilde{W}_s$) of 
$\varOmega[C]$ (resp. $\varOmega[\widetilde{C}]$).

\begin{Thm}
\begin{enumerate}
\item[{\rm (1)}] ${\mathcal C}'$ on $\varOmega[C]$ sends $W_s$ to $W_{m'_s}$.
Moreover, in {\sl Case 1} (i.e., $N>\deg(f)$) we have
\[
{\mathcal C}'\omega_{s,j} = \sum_{l=0}^{m'_s-2} \gamma_{s,(l+1)p-j}^{1/p} \cdot \omega_{m'_s, l+1}.
\]
\item[{\rm (2)}] ${\mathcal C}'$ on $\varOmega[\widetilde{C}]$ sends $\widetilde{W}_s$ to $\widetilde{W}_{m'_s}$. Moreover, we have
\[
{\mathcal C}'\omega_{s,j} = \sum_{l=0}^{(m'_s\deg(f)-(N,\deg(f))/N -1} \gamma_{s,(l+1)p-j}^{1/p} \cdot \omega_{m'_s, l+1}.
\]
\end{enumerate}
\end{Thm}
\begin{proof}
Note that $\mathcal{C}'$ sends $\omega_{s,j}$ to $\sum_l \gamma_{s,(l+1)p-j}^{1/p}\frac{x^l}{y^{m'_s}} dx$, where $l$ runs between $\left\lceil \frac{j}{p} -1 \right\rceil$ and $\left\lfloor \frac{n'_s \deg(f) + j}{p} -1\right\rfloor$.
Hence the space of the character $\zeta\mapsto \zeta^s$ for $\zeta\in \mu_N$
is sent by $\mathcal{C}'$ to that of the character $\zeta\mapsto \zeta^{m'_s}$.

(1) First, consider {\sl Case 1} (i.e., $N-\sum A_k > 0$) of $\varOmega[C]$.
In this case $\omega_{s,j}= x^{j-1}dx/y^s$ for $0\le s < N$ and $1\le j < s$ give a basis of $\varOmega[C]$.
We have to show $l+1 < m'_s$. This follows from
\[
l+1\le \left\lfloor\frac{{n'_s}\deg(f)+j}{p}\right\rfloor 
\le \left\lfloor\frac{{n'_s}N+s-1}{p}\right\rfloor
\le \left\lfloor\frac{m'_s p -1}{p}\right\rfloor < m'_s.
\]

(2) Consider the case of $\varOmega[{\widetilde C}]$.
The set of $\widetilde\pi^*\omega_{s,j}$ with
\[
	\omega_{s,j} = \frac{x^{j-1}dx}{y^s}, \quad 1 \leq s \leq N-1,\ 1 \leq j \leq \frac{s\deg(f) - (N,N-\sum{A_k})}{N}
\]
is a basis of $\varOmega[\widetilde{C}]$ by Corollary \ref{cor:basis_for_tilde_C}. By Corollary \ref{cor:basis_for_tilde_C}, we need to show that
\[
	1 \leq m'_s \leq N-1, \quad 1\leq l+1 \leq \frac{m'_s\sum{A_k}-(N,N-\sum{A_k})}{N}.
\]
The former is clear. In addition $m'_s p - n'_s N = s$ and $jN \leq s\sum{A_k} - (N,N-\sum{A_k})$, then we have
\[
	\left\lfloor\frac{n'_s\sum{A_k}+j}{p}\right\rfloor \leq \left\lfloor\frac{m'_s\sum{A_k}-(N,N-\sum{A_k})}{N} + \frac{(N,N-\sum{A_k})}{N}\frac{p-1}{p}\right\rfloor.
\]
We rewrite the first term of right side as the irreducible fraction, then the denominator is a divisor of $N/(N,N-\sum{A_k})$, while the second term of right side is strictly smaller than $(N,N-\sum{A_k})/N$. Hence we obtain
\[
	 \left\lfloor\frac{{n'_s}\sum{A_k}+j}{p}\right\rfloor \leq \frac{{m'_s}\sum{A_k}-(N,N-\sum{A_k})}{N}.
\]
This is the desired conclusion.
\end{proof}

\begin{Exp}
We consider the curve $C: y^3=x(x-1)^2(x-z)^2$ in Examples \ref{Example4.7} and \ref{Example5.4}, i.e., $(A_0,A_1,A_2)=(1,2,2)$ and $N=3$. Let us describe the Cartier operator on $\varOmega[\widetilde C]$.
The basis is given by $\frac{x^{j-1}}{y^s}dx$ 
for $(s,j)=(1,1),(2,1),(2,2),(2,3)$, see Example \ref{Example5.4}.
\begin{itemize}
\item If $p \equiv 1~({\rm mod}\,3)$, we have $m'_s = s$ and $n'_s = (p-1)s/3$.
\item If $p \equiv 2~({\rm mod}\,3)$, we have $m'_1 = 2$ and $n'_1 = (2p-1)/3$
and $m'_2 = 1$ and $n'_2 = (p-2)/3$.
\end{itemize}
Put $a'=5s/3-j$ and $c'=4s/3-j+1$, which are positive and set $d'=5n'_s-ip+j$. We have
\[
	\gamma_{s, i p-j} = \frac{(c';d')}{(a';d')}{\mathcal F}^{(d';n'_s,n'_s,n'_s,n'_s)}(a',s/3,s/3,s/3,c';1,z,z).
\]
for $i=1$ if $m'_s=1$ and $i=1,2,3$ if $m'_s=2$.
For example for $p=7$, the Cartier-Manin matrix with respect to
$\omega_{1,1},\omega_{2,1},\omega_{2,2},\omega_{2,3}$
is given by
\[
\begin{pmatrix}
z^4 + 2z^3 + z^2 + 2z + 1 &   0   &      0      & 0\\
0 & z^7  &-z^8-z^7    & z^8\\
0 &-z^7-1&z^8+z^7+z+1 & -z^8-z\\
0 &1     &-z-1        & z\
\end{pmatrix}.
\]
For example for $p=5$, the Cartier-Manin matrix is given by
\[
\begin{pmatrix}
0 &   3z + 3   &      z^2 + 4z + 1     & 3z^2 + 3z\\
4z^6 + 4z^5 &   0   &      0      & 0\\
z^6 + z^5 + z + 1 &   0   &      0      & 0\\
4z + 4 &   0   &      0      & 0\\
\end{pmatrix}.
\]
\end{Exp}

\renewcommand{\thesection}{\Alph{section}}
\setcounter{section}{0}
\section{Results from elementary number theory}
In this section, we prove some propositions used in Sections 4-6.
Let $p$ and $q$ be co-prime positive integers.
Let $g$ be a natural number.
\begin{Lem}
Every integer $d$ with $d \geq pq+1$ can be written as $d = pa+qb$ for some positive integers $a,b$.
\begin{proof}
Let $r$ be the remainder of $d \geq pq+1$ divided by $q$. Now each remainder of $p, 2p,\ldots, pq$ divided by $q$ is different, so there exists $1 \leq a \leq q$ such that the remainder of $pa$ divided by $q$ is $r$. Since $d-pa$ is divided by $q$, then $d = pa +qb$ when we put its quotient $b$. 
\end{proof}
\end{Lem}
\begin{Lem}
Every integer $d$ with $d \geq gpq+1$ can be written as $d = pa+qb$ for positive integers $a,b$ in at least $g$ ways.
\begin{proof}
Since $d - (g-1)pq \geq pq+1$, we can write $d - (g-1)pq = pa + qb~(a,b \geq 1)$ by Lemma B.1. Then $d = p(a+iq) + q(b -(i-g+1)p)$ where $i = 0,\ldots, g-1$. Now $a+iq,\,b -(i-g+1)p \geq 1$.
\end{proof}
\end{Lem}
\begin{Cor}
Every integer $d$ with $d \geq gpq-p-q+1$ can be written as $d = pa+qb$ for integers $a,b \geq 0$ in at least $g$ ways.
\begin{proof}
If $d \geq gpq-p-q+1$, then $d+p+q \geq gpq + 1$ can be written as $d+p+q = pa+qb~(a,b \geq 1)$ in different $g$ ways by Lemma B.2. Thus $d = p(a-1) + q(b-1)$, so the proposition is true.
\end{proof}
\end{Cor}
\begin{Prop}\label{B4}
Let $m,n$ be integers satisfying $np+mq=1$. Then the following are true:
\begin{enumerate}
\setlength{\itemsep}{0cm}
\item Let $d$ be an integer with $d \geq gpq-p-q+1$.
For any $e\in \{0,\ldots, g-1\}$, there exists a pair $(a,b)$
of non-negative integers such that
$d = pa + qb$ and $e \equiv -ma+nb\ ({\rm mod}\ g)$.
\item Let $d = gpq-p-q$.
For any non-negative integers $a,b$ with $d = pa + qb$,
we have $-ma+nb \not\equiv m-n \pmod{g}$.
\end{enumerate}
\begin{proof}
(1)
Let $a_0$ and $b_0$ be positive integers with $d-(g-1)pq=pa_0 + q b_0$.
Set $a_i := a_0 + iq$ and $b_i := b_0+(g-1-i) p$ for $i = 0,\ldots,g-1$. 
Let $e_i$ be the element of $\{0,\ldots,g-1\}$ with $e_i \equiv -ma_i+nb_i \pmod{g}$.
It suffices to show
$e_i \not\equiv e_j \pmod{g}$ for $0\leq i < j \leq g-1$.
This follows from
\begin{eqnarray*}
	e_i - e_j = (-ma_i+nb_i) - (-ma_j+nb_j) = -m(a_i - a_j) + n(b_i-b_j)\\
	 = -mq(i-j) - np(i-j) = -(np+mq)(i-j) = j-i \not\equiv 0 \pmod{g}.
\end{eqnarray*}

(2) Let $a_0=-1$ and $b_0=p-1$, then $d-(g-1)pq=pa_0 + q b_0$. 
Set $a_i := a_0 + iq$ and $b_i := b_0+(g-1-i) p$ for $i = 0,\ldots,g$.
Any pair $(a,b)$ of non-negative integers satisfying $d=pa+qb$
is given by $(a_i,b_i)$ for some $i=1,2,\ldots,g-1$.
Let $e_i$ be the element of $\{0,\ldots,g-1\}$ with $e_i \equiv -ma_i+nb_i \pmod{g}$. It suffices to show
$e_i \not\equiv m-n \pmod{g}$ for $1\leq i \leq g-1$.
This follows from
\begin{eqnarray*}
	e_i - (m-n) &=& (-ma_i+nb_i) - (m-n) \\
&=& -m(a_i + 1) + n(b_i+1) = -m(a_i -a_0) + n(b_i - b_0+p)\\
	 &=& -mqi + np(g-i) = -npg -(np+mq)i = -npg-i \not\equiv 0 \pmod{g}.
\end{eqnarray*}
\end{proof}
\end{Prop}

The next proposition is a generalization of \cite[Lemma 3.8]{Manin},
where Manin prove the case of $g=1$.
Put $d_0=gpq-p-q$ and $e_0=m-n$.

\begin{Prop}\label{B:Duality}
Let $d$ be an integer with $0\le d \le gpq-p-q$,
and $e$ be an integer with $0\le e \le g-1$.
There does not exist a pair $(a,b)$ of non-negative integers
such that $d=pa+qb$ and $e\equiv -ma + nb \pmod{g}$ if and only if
there exists a pair $(a',b')$ of non-negative integers such that
 $d=d_0-(pa'+qb')$ and $e\equiv e_0 - (-ma' + nb') \pmod{g}$.
\end{Prop}
\begin{proof}
First we show the ``if"-part by contradiction.
Assume that there exists a pair $(a',b')$ of non-negative integers such that
 $d=d_0-(pa'+qb')$ and $e\equiv e_0 - (-ma' + nb') \pmod{g}$
and that there is a pair $(a,b)$ of non-negative integers
such that $d=pa+qb$ and $e\equiv -ma + nb \pmod{g}$.
Then $d_0=p(a+a') + q(b+b')$ and $e_0\equiv -m(a+a')+n(b+b') \pmod{g}$.
This contradicts Proposition \ref{B4} (2).

Next we show the ``only if"-part.
Assume that there does not exist a pair $(a,b)$ of non-negative integers
such that $d=pa+qb$ and $e\equiv -ma + nb \pmod{g}$.
We claim that there exists a pair $(a'',b'')$ with $0\le a'' \le gq-1$ and $b''<0$ such that $d=pa'' + q b''$ and $e\equiv -ma''+nb'' \pmod{g}$.
Indeed, let $a''_0$ be the smallest non-negative integer with
$d\equiv pa''_0 \pmod q$. Set $b''_0=(d-pa''_0)/q \in{\mathbb Z}$. 
Choose $k$ in $\{0,\ldots,g-1\}$ such that
$e\equiv -ma''_0 + nb''_0 -k \pmod{g}$.
Put $a''=a''_0+qk$ and $b''=b''_0-pk$.
Then we have $d=pa'' + q b''$ and $e\equiv -ma''+nb'' \pmod{g}$.
By the assumption, $b''<0$ has to hold. Thus the claim was proved.
The $(a'',b'')$ obtained in the claim satisfies
\begin{eqnarray*}
d_0-d &=& gpq-p-q-(pa''+qb'') = p(gq-1-a'')+q(-b''-1),\\
e_0-e &\equiv& -m(-a''-1)+n(-b''-1)
\equiv -m(gq-1-a'')+n(-b''-1) \pmod{g}.
\end{eqnarray*}
Put $a':=gq-1-a''$ and $b':=-b''-1$.
Then $a'$ and $b'$ are non-negative and satisfy
$d=d_0-(pa'+qb')$ and $e\equiv e_0 - (-ma' + nb') \pmod{g}$.
\end{proof}

\begin{Lem}\label{B5}
Let $p,q$ be co-prime positive integers and let $d$ be an integer with $0 < d < q$ such that $d \not\equiv 0~({\rm mod}\ p)$. Then there uniquely exist $a,b$ such that $d=pa-qb$ with $0<a<q,\,0<b<p$.
\begin{proof}
There exist $a_0,b_0 \in \mathbb{Z}$ such that $pa-qb = d$ by Lemma B.1, then $(a,b) = (a_0+qk,b_0+pk)$ satisfy $ pa-qb = d$ for all $k \in \mathbb{Z}$. Since $a_0 \not\equiv 0~({\rm mod}\ q)$, then we can choose $k$ such that $0 < a < q$. Now $-d<qb<pq-d$, then $-d/q<b<p-(d/q)$ and note that $0< d/q < 1$.
\end{proof}
\end{Lem}

\end{document}